\documentclass[a4paper,11pt]{article}
\usepackage{amsmath}
\usepackage{amsfonts}
\usepackage[latin9]{inputenc}
\usepackage[english]{varioref}
\usepackage{dcolumn}
\usepackage[height=22cm , width = 16cm , top = 3cm , left = 3cm, a4paper]{geometry}
\usepackage[a4paper]{geometry}
\usepackage[final]{graphicx}
\usepackage{pgfplots}

\usepackage{epsfig}
\usepackage{graphicx}
\graphicspath{{./collision 1 case/}}
\usepackage{pstricks}
\usepackage{multirow}
\usepackage{tabularx}
\usepackage{psfrag}
\usepackage{rotating}
\usepackage{booktabs}

\usepackage{delarray}
\usepackage{amssymb}
\usepackage{blindtext}
\usepackage{rotating}
\usepackage{subfigure}
\usepackage{layout}
\usepackage{amsthm}
\theoremstyle{definition}
\newtheorem{exmp}{Example}[section]
\usepackage{xcolor}
\usepackage{setspace}

{
	{
		{
			{

				\newtheorem{thm}{Theorem}[section]

				\parindent 0pt
				\newcommand{\enter}{\bigskip}

\begin{document}
\thispagestyle{empty}
\author{Sanjiv Kumar Bariwal\footnote{{\it{${}$ Corresponding author. \, Email address:}} p20190043@pilani.bits-pilani.ac.in},  Rajesh Kumar\\
\footnotesize Department of Mathematics, Birla Institute of Technology and Science, Pilani,\\ \small{ Pilani-333031, Rajasthan, India}\\
}

\title{{Non-linear collision-induced breakage equation: approximate solution and error estimation}}	
									
\maketitle

\begin{quote}
{\small {\em\bf Abstract}}:	This article aims to provide approximate solutions for the non-linear collision-induced breakage equation using two different semi-analytical schemes, i.e., variational iteration method (VIM) and optimized decomposition method (ODM). The study also includes the detailed convergence analysis and error estimation for ODM in the case of product collisional ($K(\epsilon,\rho)=\epsilon\rho$) and  breakage ($b(\epsilon,\rho,\sigma)=\frac{2}{\rho}$) kernels with an exponential decay initial condition. By contrasting estimated concentration function and moments with exact solutions, the novelty of the suggested approaches is presented considering three numerical examples. Interestingly, in one case, VIM provides a closed-form solution, however, finite term series solutions obtained via both schemes supply a great approximation for the concentration function and moments. 
	\end{quote}
\section{Introduction}
The  collision-induced breakage (fragmentation) equation which is non-linear in nature, is used in the modeling of planet formation, aerosel, miling and crushing processes \cite{brilliantov2015size,wei2019graphite,chen2020collision} to explain the mechanics of a massive number of particles splitting apart as a consequence of collisions. It is referred to as non-linear breakage due to the interaction among particles. The recognition  of a particle is defined by its size or volume. 
Cheng and Redner \cite{cheng1988scaling} used the following integro-partial differential equation to derive the collision-induced breakage equation (CBE).  It illustrates the time progession of  concentration function  $f(\varsigma,\epsilon)\geq 0$ of particles of size $\epsilon \in ]0,\infty[$ at time $\varsigma\geq 0$ and is defined by
\begin{align}\label{maineq}
\frac{\partial{f(\varsigma,\epsilon)}}{\partial \varsigma}=  \int_0^\infty\int_{\epsilon}^{\infty} K(\rho,\sigma)b(\epsilon,\rho,\sigma)f(\varsigma,\rho)f(\varsigma,\sigma)\,d\rho\,d\sigma -\int_{0}^{\infty}K(\epsilon,\rho)f(\varsigma,\epsilon)f(\varsigma,\rho)\,d\rho
\end{align}
with the given initial data 
\begin{align}\label{initial}
f(0,\epsilon)\ \ = \ \ f^{in}(\epsilon) \geq 0, \ \ \ \epsilon \in ]0,\infty[.
\end{align}
Without losing generality, $\epsilon$ and $\varsigma$ are viewed as dimensionless quantities. In equation (\ref{maineq}), the collision kernel $K(\epsilon,\rho)$ represents the collision rate for breakage events involving  $\epsilon$ and $\rho$ sizes particles and holds symmetric property with respective sizes. The term $b(\epsilon, \rho, \sigma)$ corresponds to the breakage distribution function, which indicates the rate at which the splitting of $\rho$ volume particle produces $\epsilon$ volume particles due to collision with $\sigma$.
The  function $b$ holds
$$b(\epsilon, \rho, \sigma)\neq 0\,\,\, \text{for} \hspace{0.4cm} \epsilon \in (0,\rho)\,\,\, \text{and} \hspace{0.4cm} b(\epsilon, \rho, \sigma)=0 \,\,\,\text{for} \hspace{0.4cm} \epsilon> \rho$$
as well as satisfies
$$\int_{0}^{\rho}\epsilon b(\epsilon, \rho, \sigma)\,d\epsilon=\rho, \,\, \forall\, (\rho,\sigma) \in {]0,\infty[}^{2}.$$ The above integral property concerns about the mass conservation of particles. The volume of daughter particles produced by the splitting of $\rho$ particle is equivalent to the parent particle $\rho$. The first component in the right-hand side of  Eq.(\ref{maineq}), known as the birth term, describes the addition of $\epsilon$ size particles as the consequence of collisional fragmentation between $\rho$ and $\sigma$ sizes particles. The second term is designated as death part and explains the destruction of $\epsilon$ size particles by collision with $\rho$ size particles.\enter

 Integral features like moments are essential characteristics of the concentration function $f(\varsigma, \epsilon)$ that must also be specified. The subsequent equation identifies the $j$th moment of the solution as
\begin{align}\label{momemt}
M_{j}(\varsigma)=\int_{0}^{\infty}{\epsilon}^{j}f(\varsigma, \epsilon)\,d\epsilon, \,\, j=0,1,2,....
\end{align}
The zeroth ($j = 0$) and the first ($j = 1$) moments are proportional to the system's total number of particles and volume, respectively. Additionally, the second moment ($j=2$)  articulates the energy dissipated by the system. It is obvious, due to breakage  event, $M _0(\varsigma)$ grows as time passes, while $M _2(\varsigma)$ decreases due to the development of small particles. Nevertheless, the total volume $M_1(\varsigma)$ will be maintained fixed for specific kernels throughout the collision-breaking procedure \cite{kostoglou2000study}.

\subsection{Literature review and motivation}
In the literature, we have seen  theoretical and numerical results related to the  aggregation, linear  breakage and aggregation-breakage processes in the population dynamics such as \cite{ziff1991new,dubovskii1992exact,barik2020mass,kostoglou2009sectional,kumar2015development,Kumar2014convergence,attarakih2009solution,ahmed2013stabilized,bariwalconvergence} and references therein. The CBE \cite{cheng1990kinetics,kostoglou2000study} differs from linear breakage, in which particles break spontaneously under the influence of external factors. The non-linear CBE lacks general analytic solutions, although its mathematical representation is associated with the linear equation as in \cite{cheng1990kinetics,kostoglou2000study,ernst2007nonlinear,das2020approximate,giri2021existence,giri2021weak}. Few analytical solutions are provided in reference \cite{kostoglou2000study,ernst2007nonlinear,ziff1985kinetics} for constant and multiple collision kernels and self-similar solutions are also reported in \cite{kostoglou2000study,mcgrady1987shattering}. Due to limitation of the analytical solutions, recently finite volume (FV) and finite element methods were implemented to solve the model. The authors devised a conservative FV approximation to forecast $M_{0}(\varsigma)$ and $M_{1}(\varsigma)$ in \cite{das2020approximate}. In addition, the constant-driven Monte Carlo technique was employed to validate the FV simulations. Further, in \cite{paul2023moments}, consistency and convergence analysis of discretized FV method under twice continuously differentiable kernels were proved by demonstrating Lipschitz continuity of the numerical fluxes. The authors in \cite{lombart2022fragmentation} have introduced a discontinuous Galerkin algorithm that precisely resolves the non-linear CBE on a reduced mass grid in order to account for the dust particles.\\

Nonetheless, numerical techniques have become a common method for analysing and addressing a wide range of difficult non-linear issues. Such schemes necessitate physical assumptions such as variable discretization, a set of basis functions, linearization, etc. in order to numerically approximate the solution. Today, a lot of authors have suggested alternate strategies based on iterative methodologies to obtain the solution in series forms in order to avoid these restrictions. The so-called semi-analytical procedures enable us to obtain the results analytically. To understand the beauty and novelty of such algorithms, readers are referred to homotopy perturbation method (HPM) \cite{he2003homotopy,el2005application,kaur2019analytical}, Adomian decomposition method (ADM) \cite{jiao2002aftertreatment,singh2015adomian}, VIM \cite{he1997new,abdou2005variational,wazwaz2007comparison,odibat2010study,hasseine2015analytical,hasseine2017semi}, and ODM \cite{odibat2020optimized,kaushik2023novel} which deal with the solutions for ODEs/PDEs and integro-partial differential equations including aggregation and linear breakage equations. Authors in \cite{kaur2019analytical} explained HPM for finding analytical approximated solutions for aggregation problem considering constant, sum and product kernels as well as for linear breakage model by taking linear and quadratic selection rates. The authors demonstrated convergence analysis for a particular case of aggregation and breakage kernels and discussed the implementations of ADM technique on the aggregation and breakage equations, see \cite{singh2015adomian}. To obtain the series solution, \cite{hasseine2015analytical} described the framework of ADM and VIM for the linear breakage equation in batch and continuous flow systems with assumed functional forms of breakage frequencies and daughter particle distributions.\\

In \cite{hasseine2017semi}, ADM, HPM, and VIM were applied to solve the pure aggregation and pure breakage equations. Series solutions were then compared with the analytical results for aggregation with product kernel $\epsilon\rho$, while selection rate ${\epsilon}^{k}$ and breakage kernel $\frac{k{\epsilon}^{k-2}}{\rho}$ were used in the breakage case. An exponential decay initial condition $e^{-\epsilon}$ was considered into account in each case. Additionally, it has been discovered that VIM delivers more accurate approximating results than ADM. Very recently, the authors in \cite{kaushik2023novel} demonstrated that the ODM solutions for non-linear aggregation equation with constant, sum and product aggregation kernels, are far superior to the ADM results \cite{singh2015adomian} and also exhibit fast convergence. According to the literature survey \cite{hasseine2017semi, odibat2020optimized, kaushik2023novel}, it is found that VIM and ODM are more suitable and efficient for solving non-linear ODEs/PDEs and integro PDEs including aggregation and breakage models. Since, it is not easy to compute the exact solution for non-linear CBE (\ref{maineq}-\ref{initial}) and as per our knowledge, there is no literature available on semi-analytical techniques for this model. Our first aim of this article is to implement VIM and ODM for the said equation. Moving further, theoretical convergence analysis and error estimation are studied to justify the obtained results. For the numerical validation, three examples are selected and exact solutions for concentration and moments are compared with the approximated results. It is observed that VIM enjoys better approximation over ODM in every instance.\\

This article is framed as follows: Section \ref{analyticalmethods} contains information on the preliminary steps for VIM and ODM, as well as their applications on CBE (\ref{maineq}-\ref{initial}). The theoretical results on convergence and error analysis are presented in Section \ref{convergenceanalysis}. Section \ref{numericaldiscussion} yields the numerical implementation and comparison between ODM, VIM, and exact solutions by providing the error graphs and tables. Finally, Section \ref{conclusions} provides the conclusions.
\section{Semi-analytical Approach} \label{analyticalmethods}
\subsection{Variational iteration method}
VIM \cite{odibat2010study} is used to solve variety of non-linear  ODEs/PDEs without linearization and small perturbations. This approach  is feasible and  effective  for dealing with the non-linear problems. To understand the general idea of the VIM for any ODEs/PDEs, let us consider
\begin{align}\label{vim1}
\mathcal{L}(f(\varsigma,\epsilon))+M(f(\varsigma,\epsilon))=0,
\end{align}
where $\mathcal{L}=\frac{\partial}{\partial \varsigma}$ and $M$ are linear and non-linear operators, respectively. According to VIM, we get the correction functional for Eq.(\ref{vim1}) as
\begin{align}\label{vim2}
f_{k+1}(\varsigma,\epsilon)=f_k(\varsigma,\epsilon)+\int_{0}^{\varsigma}\left[\lambda\Big(\mathcal{L}(f_k(\tau,\epsilon))+M(\tilde{f}_k(\tau,\epsilon))\Big)\right] d \tau.
\end{align}
Eq.(\ref{vim2}) has a general Lagrange multiplier $\lambda=-1$ that can be found optimally via variational theory, and  $\tilde{f}_k$ is a
restricted variation which means $\delta \tilde{f}_k=0,$ see \cite{odibat2010study} for more details. This  provides the following iteration formula
\begin{align}\label{vim3}
f_{k+1}(\varsigma,\epsilon)=f_k(\varsigma,\epsilon)-\int_{0}^{\varsigma}\left[\Big(\mathcal{L}(f_k(\tau,\epsilon))+M(f_k(\tau,\epsilon))\Big)\right] d \tau.
\end{align}
Let us construct the operator $A[f]$ as follows:
\begin{align}\label{vim4}
A[f]=\int_{0}^{\varsigma}\left[-\Big(\mathcal{L}f(\tau,\epsilon)+Mf(\tau,\epsilon)\Big)\right] d \tau,
\end{align}
then the solution of problem (\ref{vim1}) is written in series form as $f(\varsigma,\epsilon)=\sum_{k=0}^{\infty}f_k(\varsigma,\epsilon)$ where the  components $f_k(\varsigma,\epsilon)$ for $ k=0,1,2,..., $ are defined as
 \begin{equation}\label{solutionterm1}
	  \left.\begin{aligned}
	         f_0(\varsigma,\epsilon)&=f^{in}(\epsilon),\\
	          f_1(\varsigma,\epsilon)&=A[f_0(\varsigma,\epsilon)],\\
	            f_{k+1}(\varsigma,\epsilon)&=A[f_0(\varsigma,\epsilon)+f_1(\varsigma,\epsilon)+...+ f_k(\varsigma,\epsilon)], \,\, k\geq 1.
	        \end{aligned}
	  \right\}
	 	 \end{equation}
An $n$th-order truncated series solution of the result $f(\varsigma,\epsilon)$ is considered as $\varphi_n(\varsigma,\epsilon)$ and is denoted by
	 \begin{align}\label{truncated}
	 \varphi_n(\varsigma,\epsilon):=\sum_{k=0}^{n}f_ k(\varsigma,\epsilon).	
	 \end{align} 
	 
	  In the following, the general formulation for the $n$-term series solution for the CBE (\ref{maineq}-\ref{initial}) is described. 
	  \subsubsection{VIM for CBE}
	 	Define the non-linear operator $M$ as follows
	 	 \begin{align}\label{cvim1}
	 	 Mf(\varsigma,\epsilon)=-\biggl(\int_0^\infty\int_{\epsilon}^{\infty} K(\rho,\sigma)b(\epsilon,\rho,\sigma)f(\varsigma,\rho)f(\varsigma,\sigma)\,d\rho\,d\sigma -\int_{0}^{\infty}K(\epsilon,\rho)f(\varsigma,\epsilon)f(\varsigma,\rho)\,d\rho\biggl).
	 	 \end{align}
	 	 Using the Eqs.(\ref{vim3}), (\ref{cvim1}) and $\lambda=-1$, we get the following  iteration formula to compute the concentration function 
\begin{align}\label{cvim2}
	 	 f_{k+1}(\varsigma,\epsilon)=& f_k(\varsigma,\epsilon)+\int_{0}^{\varsigma}\biggl(-\frac{\partial f_k(\tau,\epsilon)}{\partial \tau}+\int_0^{\infty}\int_{\epsilon}^{\infty} K(\rho,\sigma)b(\epsilon,\rho,\sigma)f_k(\tau,\rho)f_k(\tau,\sigma)\,d\rho\,d\sigma \nonumber \\ \nonumber 
	 	 &-\int_{0}^{\infty}K(\epsilon,\rho)f_k(\tau,\epsilon)f_k(\tau,\rho)\,d\rho\biggl) d\tau\nonumber \\
	 =	&f_k(\varsigma,\epsilon)+A[f].
	 	 \end{align}
%	 	 Opertaor $A[f]$ is
%	 	 \begin{align}\label{Aoperator}
%	 	 A[f]=&\int_{0}^{\varsigma}\biggl(-\frac{\partial f_k(\tau,\epsilon)}{\partial \tau}+\int_0^\infty\int_{\epsilon}^{\infty} K(\rho,\sigma)b(\epsilon,\rho,\sigma)f_k(\tau,\rho)f_k(\tau,\sigma)\,d\rho\,d\sigma\\ \nonumber &-\int_{0}^{\infty}K(\epsilon,\rho)f_k(\tau,\epsilon)f_k(\tau,\rho)\,d\rho\biggl) d \tau.
%	 	 \end{align}
	 	Using the above, the solution components are derived from the Eq.(\ref{solutionterm1}) and hence, an approximate series solution of $n$th-order is formed using $n$+1 components. We will simplify the details for various specific kernels in the numerical section.
\subsection{Optimized decomposition method}	
This section contains a detailed explanation of ODM \cite{odibat2020optimized} for analytically solving Eqs.(\ref{maineq}-\ref{initial}). Before going into this, let us address the essential concept of the proposed strategy for solving the following non-linear PDE
\begin{align}{\label{eq.3}}
\frac{\partial}{\partial \varsigma}f(\varsigma,\epsilon)=M[f(\varsigma,\epsilon)],
\end{align}
where $M$ is a non-linear operator of $f$ and $\mathcal{L}=\frac{\partial}{\partial \varsigma}.$ Applying inverse operator of $\mathcal{L}$ on (\ref{eq.3}) leads to the  following  equation
\begin{align}\label{4}
f(\varsigma,\epsilon)=f^{in}(\epsilon)+{\mathcal{L}}^{-1}\{M[f(\varsigma,\epsilon)]\}.
\end{align}
The fundamental concept of ODM lies in the linear approximation of non-linear term by a first-order Taylor series expansion at $\varsigma=0$. Thus, the obtained approximation is 
\begin{align}\label{5}
\frac{\partial}{\partial \varsigma}f(\varsigma,\epsilon)-M[f(\varsigma,\epsilon)] \approx \frac{\partial}{\partial \varsigma}f(\varsigma,\epsilon)-C(\epsilon)f,
\end{align}	
	where
	\begin{align}\label{cvalue}
	C(\epsilon)=\frac{\partial M}{\partial f}\Bigg\vert_{\varsigma=0}.
	\end{align}	
	 The approximation mentioned above yields, a linear operator $T$ specified by
	\begin{align}\label{Lvalue}
	T[f(\varsigma,\epsilon)]=M[f(\varsigma,\epsilon)]-C(\epsilon)f(\varsigma,\epsilon).
	\end{align}
	 The solution is now constructed as an infinite series, by following \cite{odibat2020optimized}, and  is given by
	 \begin{align}\label{fvalue}
	 f(\varsigma,\epsilon)=\sum_{k=0}^{\infty}f_{k}(\varsigma,\epsilon),
	 \end{align}
	 where the components are listed below 
	 \begin{equation}\label{solutionterm2}
	  \left.\begin{aligned}
	         f_0(\varsigma,\epsilon)&=f^{in}(\epsilon),\\
	          f_1(\varsigma,\epsilon)&={\mathcal{L}}^{-1}[Q_0(\varsigma,\epsilon)],\\
	           f_2(\varsigma,\epsilon)&={\mathcal{L}}^{-1}[Q_1(\varsigma,\epsilon)-C(\epsilon)f_1(\varsigma,\epsilon)],\\
	            f_{k+1}(\varsigma,\epsilon)&={\mathcal{L}}^{-1}[Q_k(\varsigma,\epsilon)-C(\epsilon)(f_k(\varsigma,\epsilon)-f_{k-1}(\varsigma,\epsilon))], \,\, k\geq 2,
	        \end{aligned}
	  \right\}
	 	 \end{equation}
	 	 with
	 	\begin{align}\label{Qvalue}
	 	Q_k(\varsigma,\epsilon)=\frac{1}{k!}\frac{d^k}{d{\theta}^k}\left[M\Big(\sum_{i=0}^{k}{\theta}^{i}f_i(\varsigma,\epsilon)\Big)\right],
	 	 \end{align}
	 	and 
	 	 \begin{align}
	 	 M \Big(\sum_{k=0}^{\infty}f_k(\varsigma,\epsilon)\Big)=\sum_{k=0}^{\infty}	Q_k(\varsigma,\epsilon).
	 	 \end{align}
	
\subsubsection{ODM for CBE}
	To obtain a ODM formulation for Eqs.(\ref{maineq}-\ref{initial}), consider the non-linear operator $ M[f(\varsigma,\epsilon)]$ as 
	 \begin{align}\label{codm1}
	 M[f(\varsigma,\epsilon)]=\int_0^\infty\int_{\epsilon}^{\infty} K(\rho,\sigma)b(\epsilon,\rho,\sigma)f(\varsigma,\rho)f(\varsigma,\sigma)\,d\rho\,d\sigma -\int_{0}^{\infty}K(\epsilon,\rho)f(\varsigma,\epsilon)f(\varsigma,\rho)\,d\rho,
	 \end{align}
	 and differentiating it with respect to $f(\varsigma,\epsilon)$, we have at $\varsigma=0,$
	 \begin{align}\label{codm2}
	 C(\epsilon)=-\int_{0}^{\infty}K(\epsilon,\rho)f(0,\rho)\,d\rho.
	 \end{align}
	The linear operator $T$ in Eq.(\ref{Lvalue})  expresses the following equation after utilizing the Eqs.(\ref{codm1}) and (\ref{codm2})
	\begin{align}
	 T[f(\varsigma,\epsilon)]=&\int_0^\infty\int_{\epsilon}^{\infty} K(\rho,\sigma)b(\epsilon,\rho,\sigma)f(\varsigma,\rho)f(\varsigma,\sigma)\,d\rho\,d\sigma -\int_{0}^{\infty}K(\epsilon,\rho)f(\varsigma,\epsilon)f(\varsigma,\rho)\,d\rho \nonumber\\
	 &+f(\varsigma,\epsilon)\int_{0}^{\infty}K(\epsilon,\rho)f(0,\rho)\,d\rho.
	\end{align}
	Using Eq.(\ref{codm1}) and setting $k = 0$ in Eq.(\ref{Qvalue}), the term $Q_0$ is expressed as
\begin{align*}
Q_0(\varsigma,\epsilon)=\int_0^\infty\int_{\epsilon}^{\infty} K(\rho,\sigma)b(\epsilon,\rho,\sigma)f_0(\varsigma,\rho)f_0(\varsigma,\sigma)\,d\rho\,d\sigma -\int_{0}^{\infty}K(\epsilon,\rho)f_0(\varsigma,\epsilon)f_0(\varsigma,\rho)\,d\rho,
\end{align*}
therefore,
 \begin{align}\label{f1}
 f_1(\varsigma,\epsilon)={\mathcal{L}}^{-1}\Big(\int_0^\infty\int_{\epsilon}^{\infty} K(\rho,\sigma)b(\epsilon,\rho,\sigma)f_0(\varsigma,\rho)f_0(\varsigma,\sigma)\,d\rho\,d\sigma -\int_{0}^{\infty}K(\epsilon,\rho)f_0(\varsigma,\epsilon)f_0(\varsigma,\rho)\,d\rho\Big).
 \end{align}
For $k=1$, Eqs.(\ref{Qvalue}) and  (\ref{codm1}) yield
\begin{align*}
Q_1(\varsigma,\epsilon)=\int_0^\infty\int_{\epsilon}^{\infty} K(\rho,\sigma)b(\epsilon,\rho,\sigma)\Big(f_0(\varsigma,\rho)f_1(\varsigma,\sigma)+f_1(\varsigma,\rho)f_0(\varsigma,\sigma)\Big)\,d\rho\,d\sigma\\ -\int_{0}^{\infty}K(\epsilon,\rho)\Big(f_0(\varsigma,\epsilon)f_1(\varsigma,\rho)+f_1(\varsigma,\epsilon)f_0(\varsigma,\rho)\Big)\,d\rho,
\end{align*}
and hence, with assistance of Eq.(\ref{codm2}), it provides
\begin{align}\label{f2}
 f_2(\varsigma,\epsilon)=&{\mathcal{L}}^{-1}\biggl(\int_0^\infty\int_{\epsilon}^{\infty} K(\rho,\sigma)b(\epsilon,\rho,\sigma)\Big(f_0(\varsigma,\rho)f_1(\varsigma,\sigma)+f_1(\varsigma,\rho)f_0(\varsigma,\sigma)\Big)\,d\rho\,d\sigma \nonumber\\ 
  &-\int_{0}^{\infty}K(\epsilon,\rho)\Big(f_0(\varsigma,\epsilon)f_1(\varsigma,\rho)+f_1(\varsigma,\epsilon)f_0(\varsigma,\rho)\Big)\,d\rho\nonumber \\ 
    &+f_1(\varsigma,\epsilon)\int_{0}^{\infty}K(\epsilon,\rho)f(0,\rho)\,d\rho\biggl).
\end{align}
For $k\geq 2$ and only when $i+j=k$, we have
\begin{align*}
Q_k(\varsigma,\epsilon)=&\int_0^\infty\int_{\epsilon}^{\infty} K(\rho,\sigma)b(\epsilon,\rho,\sigma)\Big(\sum_{i=0}^{k}f_i(\varsigma,\rho)\sum_{j=0}^{k}f_j(\varsigma,\sigma)\Big)\,d\rho\,d\sigma\\  &-\int_{0}^{\infty}K(\epsilon,\rho)\Big(\sum_{i=0}^{k}f_i(\varsigma,\epsilon)\sum_{j=0}^{k}f_j(\varsigma,\rho)\Big)\,d\rho,
\end{align*}
and 
\begin{align}\label{fk}
 f_{k+1}(\varsigma,\epsilon)=&{\mathcal{L}}^{-1}\biggl(\int_0^\infty\int_{\epsilon}^{\infty} K(\rho,\sigma)b(\epsilon,\rho,\sigma)\Big(\sum_{i=0}^{k}f_i(\varsigma,\rho)\sum_{j=0}^{k}f_j(\varsigma,\sigma)\Big)\,d\rho\,d\sigma\nonumber\\  &-\int_{0}^{\infty}K(\epsilon,\rho)\Big(\sum_{i=0}^{k}f_i(\varsigma,\epsilon)\sum_{j=0}^{k}f_j(\varsigma,\rho)\Big)\,d\rho
    -C(\epsilon)[f_k(\varsigma,\epsilon)-f_{k-1}(\varsigma,\epsilon)]\biggl).
\end{align}
Hence, Eq.(\ref{solutionterm2}) leads  to  the $n$-term series solution of Eqs.(\ref{maineq}-\ref{initial}) as
\begin{align}\label{truncatedsol}
\psi_n(\varsigma,\epsilon):=\sum_{k=0}^{n} f_{k}(\varsigma,\epsilon)&=f^{in}(\epsilon)+{\mathcal{L}}^{-1}\Big(\sum_{k=1}^{n}Q_{k-1}(\varsigma,\epsilon)-C(\epsilon)f_{n-1}(\varsigma,\epsilon)\Big) \nonumber \\
&= f^{in}(\epsilon)+{\mathcal{L}}^{-1}\Big(M(\psi_{n-1}(\varsigma,\epsilon))-C(\epsilon)f_{n-1}(\varsigma,\epsilon)\Big).
\end{align}	
	\section{Convergence Analysis}\label{convergenceanalysis}
	This section has discussion about the convergence of VIM and ODM solutions towards the precise ones including the error estimates for the finite term approximated solutions. The following results guarantee the convergence of VIM and provide the worst upper bound for error considering $n$-term series solution.
	
	\begin{thm}\label{conthm1}
	Let the operator $A[f],$ mentioned in (\ref{vim4}), be defined on a Hilbert space $D$ to $D$. The series solution $f(\varsigma,\epsilon)=\sum_{k=0}^{\infty}f_k(\varsigma,\epsilon)$ converges, if $$ \|A[ f_{0}(\varsigma,\epsilon)+f_{1}(\varsigma,\epsilon)+ \cdots +f_{k+1}(\varsigma,\epsilon)]\| \leq \alpha \|A[ f_{0}(\varsigma,\epsilon)+f_{1}(\varsigma,\epsilon)+ \cdots+f_{k}(\varsigma,\epsilon) ]\|,$$ $i.e.,  \|f_{k+1}(\varsigma,\epsilon)\| \leq \alpha \|f_{k}(\varsigma,\epsilon)\|$, where $0< \alpha < 1$ and $ \forall k \in \{0\} \cup \mathbb{N} $.
	\end{thm}
	\begin{proof}
	The proof of this result is explained by Z.M. Odibat in \cite{odibat2010study}[see, Theorem  \ref{conthm1}].
	\end{proof}
	\begin{thm}\label{conthm2}
	Let the series solution $\sum_{k=0}^{\infty}f_k(\varsigma,\epsilon)$ converges to the exact solution $f(\varsigma,\epsilon),$ then the truncated solution $\varphi_n$ of $f(\varsigma,\epsilon)$ has the maximum error bound 
	\begin{align}
	\|f-\varphi_n\| \leq \frac{1}{1-\alpha}{\alpha}^{1+n}\|f^{in}\|,\quad \text{with} \,\, 0< \alpha < 1.
	\end{align}
	\end{thm}
	\begin{proof} The details of the result is provided in \cite{odibat2010study}.
	\end{proof}
	It is clear from Theorems \ref{conthm1} and \ref{conthm2} and by following the idea of \cite{odibat2010study} that $\forall i \in \{0\}\cup \mathbb{N}$, the parameters,
	\begin{equation}\label{parameter}
	\gamma_i=
	    \begin{cases}
	       \frac{\|f_{i+1}\|}{\|f_{i}\|}, & \text{if } \|f_{i}\| \neq 0\\
	        0, & \text{if } \|f_{i}\|=0,
	    \end{cases}
	\end{equation}
	provide the assurance of series convergence, when $0\leq \gamma_i <1.$ Thus,  our main aim for the convergence  in VIM is to compute the values of $\gamma_i$ and show that  $0 \leq \gamma_i <1$ holds for all $i$.\\
	
	Let us begin with the convergence analysis for ODM solution. For this, assume a Banach space $\mathbb{Y}=\big(\mathcal{C}\big([0,\Gamma]:L^{1}((0,\infty),\epsilon d\epsilon)\big),f\geq0, \|\cdot\|$\big) (see, \cite{stewart1989global}) with the following enduced norm
	\begin{align}\label{norm}
	\|f\|=\sup_{\nu \in [0,\Gamma]}\int_{0}^{\infty}\epsilon |f(\nu,\epsilon)| d\epsilon < \infty.
	\end{align}
	Eq.(\ref{maineq}) provides the new form by  using the Eqs.(\ref{4}) and (\ref{codm1}), that is 
	\begin{align}\label{operator1}
	f=\mathcal{S}f,
	\end{align}
	where $\mathcal{S}:\mathbb{Y}\rightarrow \mathbb{Y} $ is a non-linear operator given as
	\begin{align}\label{operator2}
	\mathcal{S}f=f^{in}(\epsilon)+{\mathcal{L}}^{-1}\Bigl[\int_0^\infty\int_{\epsilon}^{\infty} K(\rho,\sigma)b(\epsilon,\rho,\sigma)f(\varsigma,\rho)f(\varsigma,\sigma)\,d\rho\,d\sigma -\int_{0}^{\infty}K(\epsilon,\rho)f(\varsigma,\epsilon)f(\varsigma,\rho)\,d\rho\Bigr].
	\end{align}
	To establish the convergence result, let's first prove that the operator $\mathcal{S}$ is contractive. To do so, develop an equivalent form
	\begin{align*}
	\frac{\partial}{\partial \varsigma}[f(\varsigma,\epsilon)\exp[G(\varsigma,\epsilon,f)]]=\exp[G(\varsigma,\epsilon,f)]\int_0^\infty\int_{\epsilon}^{\infty} K(\rho,\sigma)b(\epsilon,\rho,\sigma)f(\varsigma,\rho)f(\varsigma,\sigma)\,d\rho\,d\sigma,
	\end{align*}
	where,
	\begin{align*}
	G(\varsigma,\epsilon,f)=\int_0^{\varsigma}\int_{0}^{\infty} K(\epsilon,\rho)f(\nu,\rho)\,d\rho\,d\nu.
	\end{align*}
	Hence, the equivalent operator of $\mathcal{S}$ is expressed  by $\tilde{\mathcal{S}}$ as
	\begin{align}\label{eqoperator}
	\tilde{\mathcal{S}}f=&f^{in}(\epsilon)\exp[-G(\varsigma,\epsilon,f)] \nonumber \\ 
	&+\int_{0}^{\varsigma}\exp[G(\nu,\epsilon,f)-G(\varsigma,\epsilon,f)]\int_0^\infty\int_{\epsilon}^{\infty} K(\rho,\sigma)b(\epsilon,\rho,\sigma)f(\nu,\rho)f(\nu,\sigma)\,d\rho\,d\sigma\,d\nu.
	\end{align}
	In general, it is not easy to demonstrate the contractive nature of $\tilde{\mathcal{S}}$ due to the model's complexity. However, for a specific set of kernels, in the following theorem, the self-mapping and contraction results are proved under some additional hypotheses. 
	
	\begin{thm}\label{conthm3}
	Assume that the non-linear operator $\tilde{\mathcal{S}}$ is defined in (\ref{eqoperator}) and the set $\mathbb{D}$ is introduced as $\mathbb{D}=\{f \in \mathbb{Y}: \, \|f\|\leq L\}$. If the following hypotheses
	\begin{description}
	\item[(a)]$K(\epsilon,\rho)=\epsilon\rho,\quad b(\epsilon,\rho,\sigma)=\frac{2}{\rho}, \quad \text{and} \quad f^{in}(x)=\exp(-\epsilon)\,\, \forall \, \epsilon,\rho,\sigma \in(0,\infty),$ 
	\item[(b)] $\eta:=\frac{1}{e}+\frac{3}{e}M_{2}(0)L{\varsigma_1}<1,$ and $\max(\|f^{in}\|,M_{2}(0))(1+L\varsigma_0)\leq L$ for $\varsigma_0, \varsigma_1\in [0,\Gamma]$ 
	\end{description}
	hold, then the operator $\tilde{\mathcal{S}}:\mathbb{D }\rightarrow \mathbb{D}$ and has contractive nature, i.e., $\|\tilde{\mathcal{S}}f_1-\tilde{\mathcal{S}}f_2\|\leq \eta \|f_1-f_2\|, \forall \,  (f_1,f_2) \in \mathbb{D} \times \mathbb{D}.$
	\end{thm}
	\begin{proof}
	Let us begin with the proof of $\tilde{\mathcal{S}}:\mathbb{D }\rightarrow \mathbb{D}.$ For this, multiplying Eq.(\ref{eqoperator}) with $\epsilon$ and  integrating over the domain of $\epsilon$ provide
	\begin{align*}
	\|\tilde{\mathcal{S}}\|&\leq \|f^{in}\|+\int_{0}^{\varsigma}\int_{0}^{\infty}\exp[G(\nu,\epsilon,f)-G(\varsigma,\epsilon,f)]\int_0^\infty\int_{\epsilon}^{\infty}2\epsilon\sigma f(\nu,\rho)f(\nu,\sigma)\,d\rho\,d\sigma\,d\epsilon\,d\nu\nonumber\\
	&\leq \|f^{in}\|+\int_{0}^{\varsigma}\int_{0}^{\infty}\exp\Big[-\int_{\nu}^{\varsigma}\int_{0}^{\infty}\epsilon\rho f(\xi,\rho)\,d \rho\, d \xi\Big]\int_0^\infty\int_{\epsilon}^{\infty}2\epsilon\sigma f(\nu,\rho)f(\nu,\sigma)\,d\rho\,d\sigma\,d\epsilon\,d\nu \nonumber \\
	&\leq \|f^{in}\|+\int_{0}^{\varsigma}\int_{0}^{\infty}\int_0^\infty\int_{\epsilon}^{\infty}2\epsilon\sigma f(\nu,\rho)f(\nu,\sigma)\,d\rho\,d\sigma\,d\epsilon\,d\nu.
	\end{align*}
	After changing the order of integration, we received 
	\begin{align*}
	\|\tilde{\mathcal{S}}\|&\leq \|f^{in}\|+\int_{0}^{\varsigma}\int_{0}^{\infty}\int_{0}^{\infty}\int_{0}^{\rho}2\epsilon\sigma f(\nu,\rho)f(\nu,\sigma)\,d\epsilon\,d\rho\,d\sigma\,d\nu \nonumber \\
	&\leq \|f^{in}\|+M_{2}(0)L\varsigma.
	\end{align*}
	Hence, $\|\tilde{\mathcal{S}}\|\leq L$ holds if $ \max(\|f^{in}\|,M_{2}(0))(1+L\varsigma_{0})\leq L$ is considered, 
	for a suitable parameter $\varsigma_{0} \in (0,\Gamma).$ \\
	
	Now, choose $f_1,f_2 \in \mathbb{D}.$ The expression of $\tilde{\mathcal{S}}f_1-\tilde{\mathcal{S}}f_2$ is given by
	\begin{align}\label{operator11}
	\tilde{\mathcal{S}}f_1-\tilde{\mathcal{S}}f_2=&f^{in}(\epsilon)F(0,\varsigma,\epsilon)+\int_{0}^{\varsigma}F(\nu,\varsigma,\epsilon)\int_0^\infty\int_{\epsilon}^{\infty}k(\rho,\sigma)b(\epsilon,\rho,\sigma)f_1(\nu,\rho)f_1(\nu,\sigma)\,d\rho\,d\sigma\,d\nu \nonumber\\ +\int_{0}^{\varsigma}&F(\nu,\varsigma,\epsilon)\int_0^\infty\int_{\epsilon}^{\infty}k(\rho,\sigma)b(\epsilon,\rho,\sigma)f_2(\nu,\rho)f_2(\nu,\sigma)\,d\rho\,d\sigma\,d\nu \nonumber \\
	+\int_{0}^{\varsigma}\exp&[G(\nu,\epsilon,f_2)-G(\varsigma,\epsilon,f_2)]\int_0^\infty\int_{\epsilon}^{\infty}K(\rho,\sigma)b(\epsilon,\rho,\sigma)f_1(\nu,\rho)f_1(\nu,\sigma)\,d\rho\,d\sigma\,d\nu \nonumber \\
	 -\int_{0}^{\varsigma}\exp&[G(\nu,\epsilon,f_1)-G(\varsigma,\epsilon,f_1)]\int_0^\infty\int_{\epsilon}^{\infty}K(\rho,\sigma)b(\epsilon,\rho,\sigma)f_2(\nu,\rho)f_2(\nu,\sigma)\,d\rho\,d\sigma\,d\nu,
	\end{align}
	where,
	\begin{align*}
	F(\nu,\varsigma,\epsilon)=\exp[G(\nu,\epsilon,f_1)-G(\varsigma,\epsilon,f_1)]-\exp[G(\nu,\epsilon,f_2)-G(\varsigma,\epsilon,f_2)].
	\end{align*}
	It should be mentioned here that the considered set of kernels, provide the mass conservation \cite{kostoglou2000study}, i.e., $M_1(\varsigma)=M_1(0)=1.$ An upper bound of $F$ can be obtained by using   $e^{-x}-e^{-y} \leq -e^{-x}(x-y)$ and $ xe^{-x}\leq \frac{1}{e}$ for all $x,y \in [0,\infty),$  as
	\begin{align}\label{Qbound}
	|F(\nu,\varsigma,\epsilon)|&= \bigg|\exp\Big[-\int_{\nu}^{\varsigma}\int_{0}^{\infty}\epsilon\rho f_1(\xi,\rho)\,d \rho\, d \xi\Big]-\exp\Big[-\int_{\nu}^{\varsigma}\int_{0}^{\infty}\epsilon\rho f_2(\xi,\rho)\,d \rho\, d \xi\Big]\bigg|\nonumber \\
	& \leq \exp\Big[-\int_{\nu}^{\varsigma}\int_{0}^{\infty}\epsilon\rho f_1(\xi,\rho)\,d \rho\, d \xi\Big]\epsilon(\varsigma-\nu)\|f_1-f_2\| \nonumber \\
	& = \epsilon(\varsigma-\nu)\exp[-\epsilon(\varsigma-\nu)\|f_{1}\|]\|f_1-f_2\| \nonumber \\
	&= \epsilon(\varsigma-\nu)\|f_{1}\|\exp[-\epsilon(\varsigma-\nu)\|f_{1}\|]\, \|f_1-f_2\| \leq  \frac{1}{e} \|f_1-f_2\|.
	\end{align}
	To display the contractive map of $\tilde{\mathcal{S}}:\mathbb{D }\rightarrow \mathbb{D},$ employ the norm on Eq.(\ref{operator11}) and  utilizing the values of $K,b,$ and Eq.(\ref{Qbound}), we have 
	\begin{align*}
	\|\tilde{\mathcal{S}}f_1-\tilde{\mathcal{S}}f_2\| \leq & \frac{1}{e}\|f_1-f_2\|\|f^{in}\| \nonumber \\
	&+\frac{2}{e}\|f_1-f_2\|\int_{0}^{\varsigma}\int_0^\infty\int_0^\infty\int_0^{\rho}\epsilon\sigma\{f_1(\nu,\rho)f_1(\nu,\sigma)+f_2(\nu,\rho)f_2(\nu,\sigma)\}\,d\epsilon\,d\rho\,d\sigma\,d\nu \nonumber \\
	& + 2\Bigg[\int_{0}^{\varsigma}\exp[G(\nu,\epsilon,f_1)-G(\varsigma,\epsilon,f_1)]\int_0^\infty\int_0^\infty\int_0^{\rho}\epsilon\sigma f_2(\nu,\rho) f_2(\nu,\sigma)\,d\epsilon\,d\rho\,d\sigma\,d\nu\nonumber \\
	&-\int_{0}^{\varsigma}\exp[G(\nu,\epsilon,f_2)-G(\varsigma,\epsilon,f_2)]\int_0^\infty\int_0^\infty\int_0^{\rho}\epsilon\sigma f_1(\nu,\rho)f_1(\nu,\sigma)\,d\epsilon\,d\rho\,d\sigma\,d\nu\Bigg].
	\end{align*}
	It is known from \cite{kostoglou2000study} that the second moment of Eqs.(\ref{maineq}-\ref{initial}) is $M_2(\varsigma)=\frac{2}{1+\varsigma}$ which is clearly bounded by $M_2(0).$ Using this fact and Eq.(\ref{Qbound}) as well as some further simplifications, the above inequality becomes
	\begin{align}
	\|\tilde{\mathcal{S}}f_1-\tilde{\mathcal{S}}f_2\| \leq & \frac{1}{e}\|f_1-f_2\|\|f^{in}\|+ \frac{1}{e}\|f_1-f_2\|M_{2}(0)\varsigma\{\|f_1\|+\|f_2\|\} \nonumber \\
	& + \int_{0}^{\varsigma}\exp[G(\nu,\epsilon,f_1)-G(\varsigma,\epsilon,f_1)] M_{2}(0)\|f_2\|\,d\nu\nonumber \\
	&-\int_{0}^{\varsigma}\exp[G(\nu,\epsilon,f_2)-G(\varsigma,\epsilon,f_2)] M_{2}(0)\|f_1\|\,d\nu\nonumber \\
	\leq & \frac{1}{e}\|f_1-f_2\|+\frac{2}{e}\|f_1-f_2\|L M_{2}(0)\varsigma+L M_{2}(0)\int_{0}^{\varsigma}F(\nu,\varsigma,\epsilon)\,d\nu \nonumber \\
	\leq & \Big(\frac{1}{e}+\frac{3}{e}M_{2}(0)L\varsigma\Big)\|f_1-f_2\|.
	\end{align}
	Hence, the operator $\tilde{\mathcal{S}}$ holds contractive property under the condition that $\eta: =\frac{1}{e}+\frac{3}{e}M_{2}(0)L{\varsigma}_{1}<1,$ for the selective parameter $\varsigma_1.$ Consequently, an invariant ball of radius $L$ exists for a small parameter $\varsigma=\min(\varsigma_0,\varsigma_1)$, and $\tilde{\mathcal{S}}$ is contractive within this ball.
	\end{proof}

	\begin{thm}\label{conthm4}
	Let the series solution $\sum_{k=0}^{\infty}f_k(\varsigma,\epsilon)$ converges to the exact solution $f(\varsigma,\epsilon)$. Then the truncated solution $\psi_n$ of $f(\varsigma,\epsilon)$ has the maximum error bound 
	\begin{align}\label{error}
	\|f-\psi_m\| \leq \frac{{\eta}^{m}}{1-\eta}\|f_1\|,
	\end{align}
	if the following assumptions satisfy
	\begin{description}
	\item[(A)]  $ \eta:=\frac{1}{e}+\frac{3}{e}M_{2}(0)L{\varsigma}_{0}<1,$ where $L=M_1(1+\Gamma)$ with $\Gamma \in (0,\varsigma)$ and $\|f_1\|<\infty$.
	\item[(B)]for any $0\leq l\leq m-1$, $\|f_{m-l}-f_{m-(l+1)}\|<\epsilon, $ where $\epsilon =\frac{1}{n^p}$ such that $p>1.$
	\end{description}
	
	\end{thm}
	\begin{proof}
	Initially, with assistance of Eqs.(\ref{codm1}), (\ref{truncatedsol}) and (\ref{operator2}), $n$-term truncated solution develops into
	\begin{align}
	\psi_n(\varsigma,\epsilon)=\tilde{\mathcal{S}}\psi_{n-1}(\varsigma,\epsilon)-\int_{0}^{\varsigma}C(\epsilon)f_{n-1}(\nu,\epsilon)d\nu.
	\end{align}
	Consider the term $\|\psi_{m+1}(\varsigma,\epsilon)-\psi_m(\varsigma,\epsilon)\|$ and the triangle inequality  provides
	\begin{align}\label{triangle}
	\|\psi_{m+1}(\varsigma,\epsilon)-\psi_m(\varsigma,\epsilon)\| \leq \|\tilde{\mathcal{S}}\psi_{m}-\tilde{\mathcal{S}}\psi_{m-1}\|+\|\int_{0}^{\varsigma}C(\epsilon)\big(f_{m}(\nu,\epsilon)-f_{m-1}(\nu,\epsilon)\big)d\nu\|.
	\end{align}
	Imposing the result of Theorem \ref{conthm3} and hypothesis ({\textbf{B}}) lead to
	\begin{align*}
	\|\psi_{m+1}(\varsigma,\epsilon)-\psi_m(\varsigma,\epsilon)\| \leq \eta \|\psi_{m}-\psi_{m-1}\|+\epsilon \, \varsigma.
	\end{align*}
	The overhead inequality is converted into the following one after doing some further simplifications 
	\begin{align*}
	\|\psi_{m+1}(\varsigma,\epsilon)-\psi_{m}(\varsigma,\epsilon)\| \leq &\eta \|\psi_{m}-\psi_{m-1}\|+\epsilon \, \varsigma \\
	\leq & \eta\Big( \|\tilde{\mathcal{S}}\psi_{m-1}-\tilde{\mathcal{S}}\psi_{m-2}\|+\|\int_{0}^{\varsigma}C(\epsilon)\big(f_{m-1}(\nu,\epsilon)-f_{m-2}(\nu,\epsilon)\big)d\nu\|\Big)+\epsilon \, \varsigma \\
	 \leq & \eta (\eta \|\psi_{m-1}-\psi_{m-2}\|+\epsilon \, \varsigma)+\epsilon \, \varsigma \\
	 \vdots &\\
	\leq & {\eta}^{m}\|\psi_{1}-\psi_{0}\|+\epsilon \, \varsigma(1+\eta +{\eta}^2+\cdots+{\eta}^{m-1} ).
	\end{align*}
	Hence, the triangle inequality helps to attain the following result
	\begin{align*}
	\|\psi_{n}(\varsigma,\epsilon)-\psi_{m}(\varsigma,\epsilon)\| \leq& \|\psi_{n}(\varsigma,\epsilon)-\psi_{n-1}(\varsigma,\epsilon)\|+\|\psi_{n-1}(\varsigma,\epsilon)-\psi_{n-2}(\varsigma,\epsilon)\| \\
	&+\cdots+\|\psi_{m+1}(\varsigma,\epsilon)-\psi_{m}(\varsigma,\epsilon)\|\\
	 \leq& \big[{\eta}^{n-1}\|\psi_{1}-\psi_{0}\|+\epsilon \, \varsigma(1+\eta +{\eta}^2+\cdots+{\eta}^{n-2} )\big]\\
	 &+ \big[{\eta}^{n-2}\|\psi_{1}-\psi_{0}\|+\epsilon \, \varsigma(1+\eta +{\eta}^2+\cdots+{\eta}^{n-3} )\big]+...\\
	 &+\big[{\eta}^{m}\|\psi_{1}-\psi_{0}\|+\epsilon \, \varsigma(1+\eta +{\eta}^2+\cdots+{\eta}^{m-1} )\big]\\
	 =&({\eta}^{n-1}+{\eta}^{n-2}+\cdots+{\eta}^{m})\|\psi_{1}-\psi_{0}\|+\epsilon \, \varsigma[(1+\eta +{\eta}^2+\cdots+{\eta}^{n-2} )\\
	 &+(1+\eta +{\eta}^2+\cdots+{\eta}^{n-3} )+\cdots + (1+\eta +{\eta}^2+\cdots+{\eta}^{m-1} )].
	\end{align*}
	Further, simplifying  the above equation, the assumption $({\textbf{A}})$  and  a desirable $\varsigma_{0}$ give us
	\begin{align}\label{error1}
	\|\psi_{n}(\varsigma,\epsilon)-\psi_{m}(\varsigma,\epsilon)\| =& \frac{{\eta}^{m}{(1-\eta)}^{n-m}}{1-\eta}\|\psi_{1}-\psi_{0}\|+\epsilon \, \varsigma\Big[\frac{{(1-\eta)}^{n-1}}{1-\eta}+\frac{{(1-\eta)}^{n-2}}{1-\eta}+\cdots +\frac{{(1-\eta)}^{m}}{1-\eta}\Big] \nonumber \\
	 \leq& \frac{{\eta}^{m}}{1-\eta}\|f_1\|+\frac{\epsilon \varsigma_{0} }{1-\eta}(n-m).
	\end{align}
	Thanks to the assumption $({\textbf{B}})$ and  $\frac{1}{m^p}>\frac{1}{n^p}$, Eq.(\ref{error1}) converges to zero as $m\rightarrow \infty$. It indicates that $\lim_{n\rightarrow \infty}\psi_{n}(\varsigma,\epsilon)=\psi(\varsigma,\epsilon),$  therefore, $f(\varsigma,\epsilon)=\sum_{k=0}^{\infty}f_{k}(\varsigma,\epsilon)=\lim_{n\rightarrow \infty}\psi_{n}(\varsigma,\epsilon)=\psi(\varsigma,\epsilon),$ which is the exact solution of (\ref{operator1}). Hence, the maximum error bound (\ref{error}) is attained through the Eq.(\ref{error1}), for a fixed $m$ and letting $n\rightarrow \infty.$
	\end{proof}
	\section{Numerical Discussion}\label{numericaldiscussion}
	This section portrays the implementation of VIM and ODM on CBE with three test cases. MATHEMATICA software displays all the computations, results, and graphs for the concentration and moments of the problems. To see  the efficiency and accuracy of our proposed methods, analytical solutions for concentration and moments are compared with the finite term  VIM and ODM results. In the absence of a precise solution, the difference between successive series terms is provided to justify the convergence of the method.
	\begin{exmp}
	Assuming Eqs.(\ref{maineq}-\ref{initial}) with  $K(\epsilon,\rho)=\epsilon\rho$, $b(\epsilon,\rho,\sigma)=\frac{2}{\rho}$ and $f^{in}(\epsilon)=\exp(-\epsilon)$ for which the  corresponding exact solution is  $f(\varsigma,\epsilon)=(1+\varsigma)^{2}\exp(-\epsilon(1+\varsigma))$  described in \cite{kostoglou2000study}. 
	\end{exmp}
	The VIM approach is employed for this example and we get the operator $A$  with the  help of Eq.(\ref{cvim2})
	 \begin{align*}
		 A[f]=&\int_{0}^{\varsigma}\biggl(-\frac{\partial f_k(\tau,\epsilon)}{\partial \tau}+\int_0^\infty\int_{\epsilon}^{\infty} 2\sigma f_k(\tau,\rho)f_k(\tau,\sigma)\,d\rho\,d\sigma \nonumber -\int_{0}^{\infty}\epsilon\rho f_k(\tau,\epsilon)f_k(\tau,\rho)\,d\rho\biggl) d \tau,
		 \end{align*}
		 and Eq.(\ref{solutionterm1}) assist for providing the following series terms
		 \begin{align*}
		 f_{0}(\varsigma,\epsilon)=&e^{-\epsilon},\quad
		      f_{1}(\varsigma,\epsilon)=\varsigma(2{e}^{-\epsilon}-{e}^{-\epsilon}\epsilon), \quad
		       f_{2}(\varsigma,\epsilon)={e}^{-\epsilon}{\varsigma}^{2}-2{e}^{-\epsilon}{\varsigma}^{2}\epsilon+\frac{1}{2}{e}^{-\epsilon}{\varsigma}^{2}{\epsilon}^2,\\
		    f_{3}(\varsigma,\epsilon)=&-{e}^{-\epsilon}{\varsigma}^{3}\epsilon+{e}^{-\epsilon}{\varsigma}^{3}{\epsilon}^{2}-\frac{1}{6}{e}^{-\epsilon}{\varsigma}^{3}{\epsilon}^3,\quad
		    f_{4}(\varsigma,\epsilon)=\frac{1}{2}{e}^{-\epsilon}{\varsigma}^{4}{\epsilon}^{2}-\frac{1}{3}{e}^{-\epsilon}{\varsigma}^{4}{\epsilon}^{3}+\frac{1}{24}{e}^{-\epsilon}{\varsigma}^{4}{\epsilon}^4,\\
		    f_{5}(\varsigma,\epsilon)=&-\frac{1}{6}{e}^{-\epsilon}{\varsigma}^{5}{\epsilon}^{3}+\frac{1}{12}{e}^{-\epsilon}{\varsigma}^{5}{\epsilon}^{4}-\frac{1}{120}{e}^{-\epsilon}{\varsigma}^{5}{\epsilon}^5
		 \end{align*}
		 and so on.
		 The further terms can be computed using Eq.(\ref{solutionterm1}) and it is easy to see that the series solution of $n$-terms, i.e., $ \varphi_n(\varsigma,\epsilon)=\sum_{k=0}^{n}f_ k(\varsigma,\epsilon)$ is  written as
		  \begin{align}\label{nsolution}
		 \varphi_n(\varsigma,\epsilon)=f_{0}(\varsigma,\epsilon)+f_{1}(\varsigma,\epsilon)+\sum_{k=2}^{n}\Bigg(\frac{(-1)^{k}{\epsilon}^{k-2}}{(k-2)!}+\frac{2(-1)^{k-1}{\epsilon}^{k-1}}{(k-1)!}+\frac{(-1)^{k}{\epsilon}^{k}}{(k)!}\Bigg)e^{-\epsilon}{\varsigma}^{k}.
		 \end{align}
		 Taking the limit $n\rightarrow \infty,$ Eq. (\ref{nsolution}) reduces to
		 \begin{align*}
		  \lim_{n \rightarrow \infty}\varphi_n(\varsigma,\epsilon)=\sum_{k=0}^{\infty} \varphi_n(\varsigma,\epsilon)=(1+\varsigma)^{2}\exp(-\epsilon(1+\varsigma)),
		 \end{align*}
	 which is exactly the precise solution.
		Now, for the  ODM approach,  Eqs.(\ref{codm1}), (\ref{codm2}) and (\ref{f1}-\ref{fk}) helped us to get the first few  series terms with $C(\epsilon)=-\epsilon,$ as
	\begin{align*}
	 f_{0}(\varsigma,\epsilon)=&e^{-\epsilon},\, 
		      f_{1}(\varsigma,\epsilon)=\varsigma(2{e}^{-\epsilon}-{e}^{-\epsilon}\epsilon),\\
		       f_{2}(\varsigma,\epsilon)=&\frac{1}{2}{\varsigma}^2(-2e^{-\epsilon}(-1+\epsilon)+e^{-\epsilon}(-2+\epsilon)\epsilon+\epsilon(2e^{-\epsilon}-e^{-\epsilon}\epsilon)),\\
		    f_{3}(\varsigma,\epsilon)=&\frac{1}{6}e^{-\epsilon}{\varsigma}^{2}(3(-2+\epsilon)\epsilon-2\varsigma(2+\epsilon)), \\
		    f_{4}(\varsigma,\epsilon)=&-\frac{1}{12}{e}^{-\epsilon}{\varsigma}^{3}(-8(1-\epsilon+{\epsilon}^2)+\varsigma(20-14\epsilon+3{\epsilon}^2)),\\
		    \intertext{and}
		    f_{5}(\varsigma,\epsilon)=&-\frac{1}{30}{e}^{-\epsilon}{\varsigma}^{3}(5(-2+\epsilon){\epsilon}^2-5\varsigma(17-8\epsilon+4{\epsilon}^2)+{\varsigma}^2(38-63\epsilon+17{\epsilon}^2)).
	\end{align*}
	Thanks to MATHEMATICA, the higher terms can be determined using (\ref{fk}) and hence, an approximate solution by taking fixed number of series terms can be taken. Fig.4.1(A) demonstrates the comparison of exact solution with $n=$4, 6, 8, 10 terms series solutions obtained via VIM and ODM at time $\varsigma=1.5$ and $\varsigma=0.6$, respectively. From these graphs, it is clear that as the number of terms increases, the series solution tends to the precise one using both methods, and in fact, the 10-term approximate solution provides excellent agreement with the exact solution. In case of ODM, it is observed that, as time increases, one needs a large number of series terms to get better accuracy, whereas VIM is consistently performing well. As expected, due to the collision event, the decreasing behaviour of the concentration function is noticed as time progresses in Fig.4.1(A). 
	 \renewcommand{\thefigure}{4.1(A)}
	\begin{figure}[htb]
	\centering
	\subfigure[ VIM at time 1.5]{\includegraphics[width=0.35\textwidth,height=0.255\textwidth]{{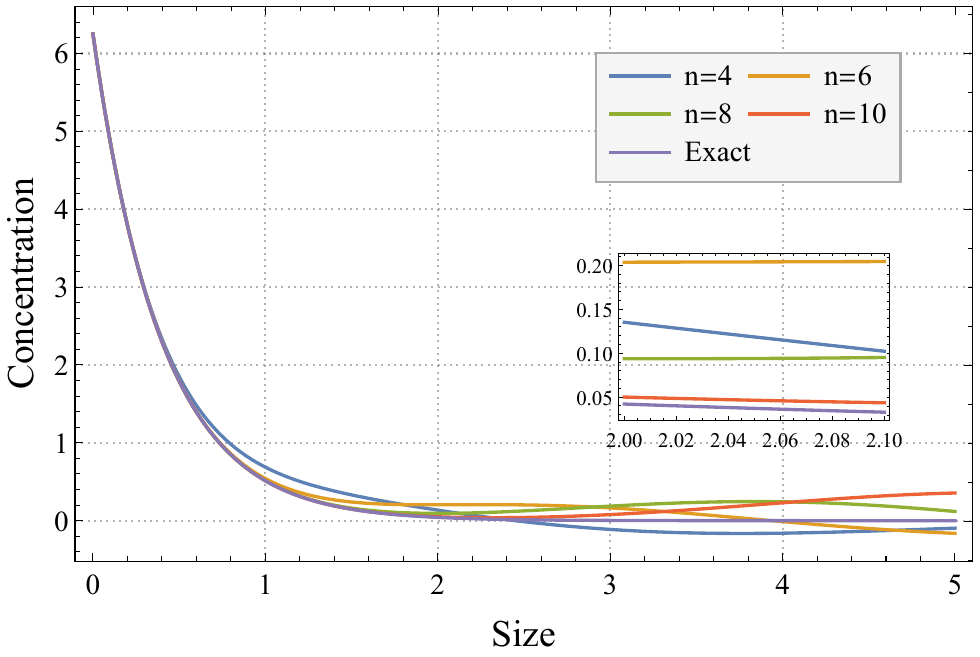}}}
	\subfigure[ ODM at time 0.6]{\includegraphics[width=0.35\textwidth,height=0.255\textwidth]{{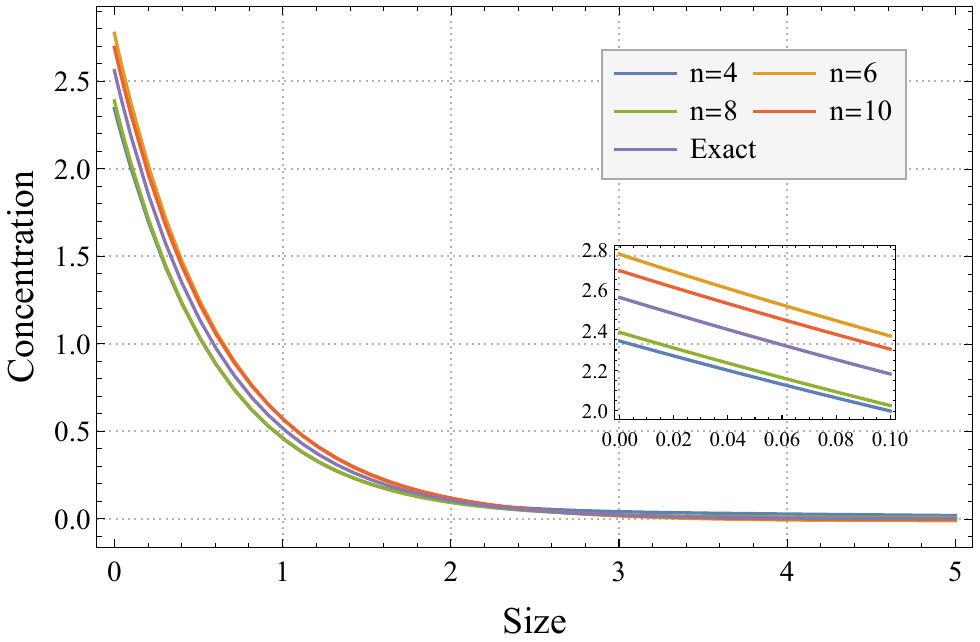}}}
	\caption{Series solutions of VIM, ODM and  exact solution}
	\label{fig0}
	\end{figure}
	Further, to see the excellency of our algorithms, the 3D display of the concentration function is visualized in Figs.4.1(B)(a) and 4.1(B)(b) for $n=10$ at time 0.6 by VIM and ODM, respectively. These plots are almost identical to the exact concentration function given in Fig.4.1(B)(c). However, as time increases to 1, the ODM solution does not have better precision with the analytical solution as compared to the VIM result, see Fig.4.1(C).\\
	
	The noticeable dissimilarity can visualize through the error plots of series and precise solutions. Fig.4.1(D)(a) portrays a minor error in approximate solution by VIM than the ODM outcome in Fig.4.1(D)(b). The same information is also conveyed through the absolute error distribution provided in Table 1 for various values of $n$ and time. The error computations are derived by dividing the size variable domain [0, 5] into subintervals $[\epsilon_{i-1/2}, \epsilon_{i+1/2}],$ where $i=1,2,\ldots,1000.$  The representative  of the $i^{th}$ cell is denoted by the mid-point $\epsilon_{i}=\frac{\epsilon_{i-1/2}+\epsilon_{i+1/2}}{2}.$ Assuming step size $h_i=\epsilon_{i+1/2}-\epsilon_{i-1/2},$ the following rule is operated to estimate the error in VIM and ODM, respectively,
	$$\text{Error}=\sum_{i=1}^{1000}|\varphi_{n}^{i}-f_{i}|h_{i},\hspace{0.5cm} \text{Error}=\sum_{i=1}^{1000}|\psi_{n}^{i}-f_{i}|h_{i},$$ 
	where $\varphi_{n}^{i}=\varphi_{n}(\varsigma,\epsilon_{i})$,\,  $\psi_{n}^{i}=\psi_{n}(\varsigma,\epsilon_{i})$ and $f_{i}=f(\varsigma,\epsilon_{i}).$
	 As expected, the error reduces with the number of terms being increased at a particular time in each case. Moreover, the VIM solution has less error by taking a fixed number of terms at each time, i.e., VIM produces a better approximate solution than the ODM. Finally, Eq.(\ref{parameter}) permits to attain the values of $\gamma_{i}$ for $i=5,6,\ldots,9$ and it is marked from Table 2 that all $\gamma_{i}<1$ which is the required condition of series convergence in VIM.\\
	 \renewcommand{\thefigure}{4.1(B)}
	\begin{figure}[htb!]
	\centering
	\subfigure[VIM Concentration ($\varphi_{10}$) ]{\includegraphics[width=0.30\textwidth,height=0.25\textwidth]{{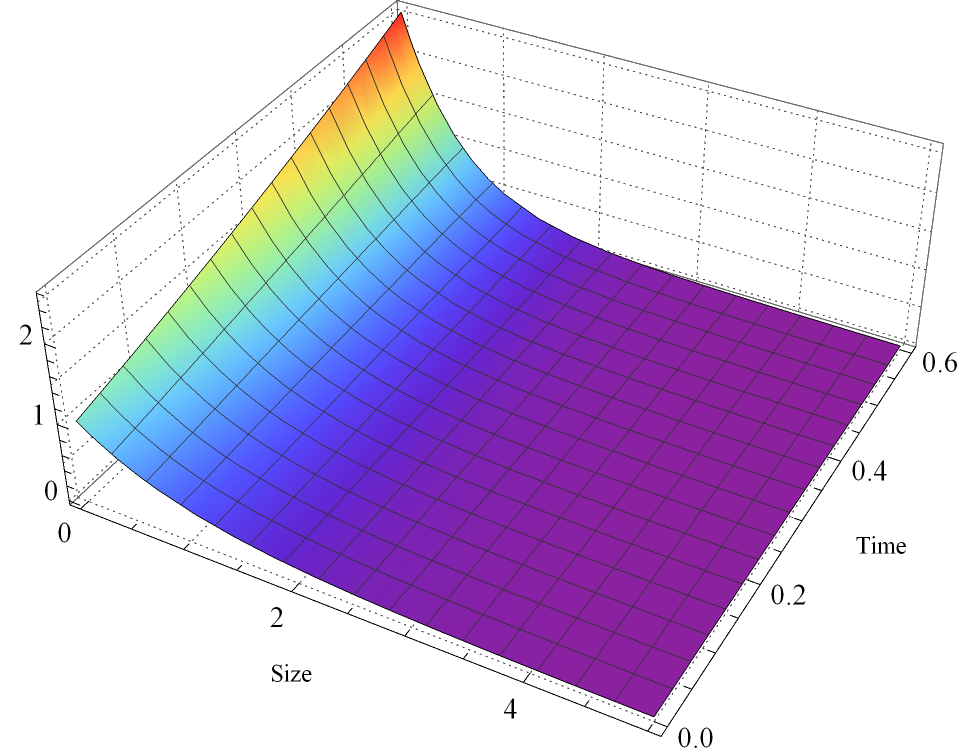}}}
	\subfigure[ODM Concentration   ($\psi_{10}$)]{\includegraphics[width=0.30\textwidth,height=0.25\textwidth]{{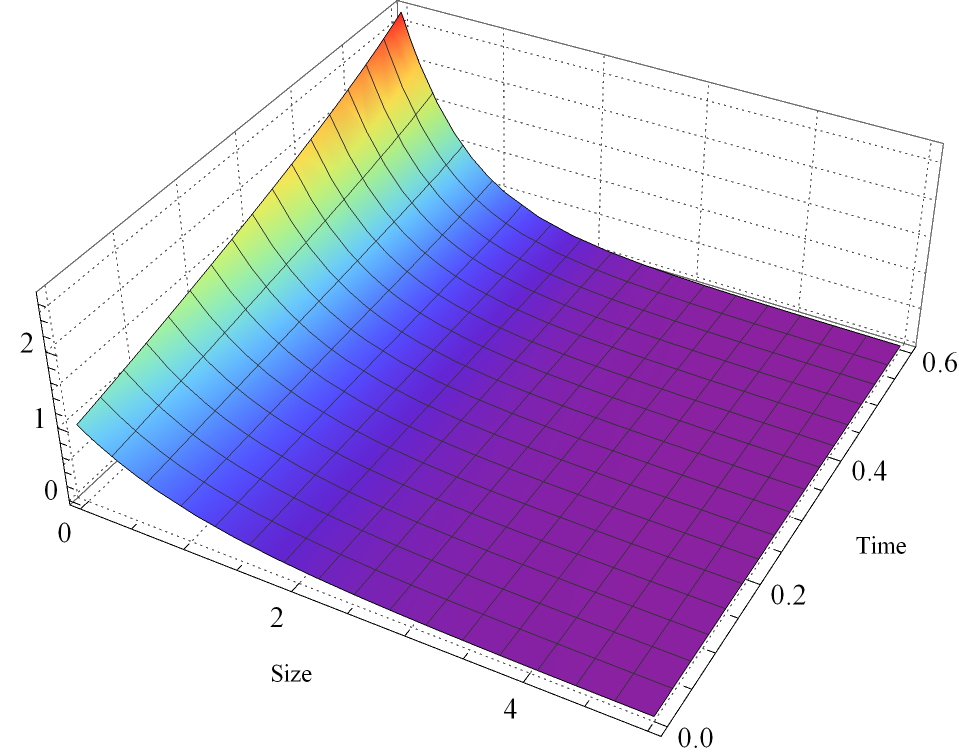}}}
	\subfigure[Exact Concentration]{\includegraphics[width=0.30\textwidth,height=0.25\textwidth]{{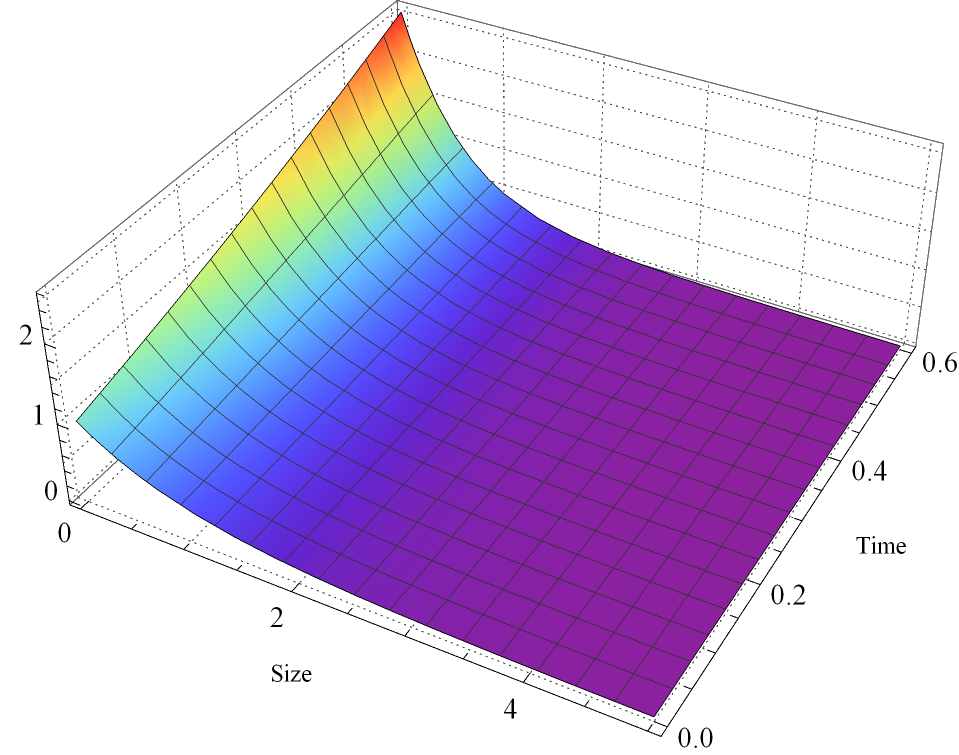}}}
	\caption{ Concentration plots at time $\varsigma$= 0.6 }
	\label{fig1}
	\end{figure}
	 \renewcommand{\thefigure}{4.1(C)}
	\begin{figure}[htb!]
	\centering
	\subfigure[VIM Concentration ($\varphi_{10}$) ]{\includegraphics[width=0.30\textwidth,height=0.25\textwidth]{{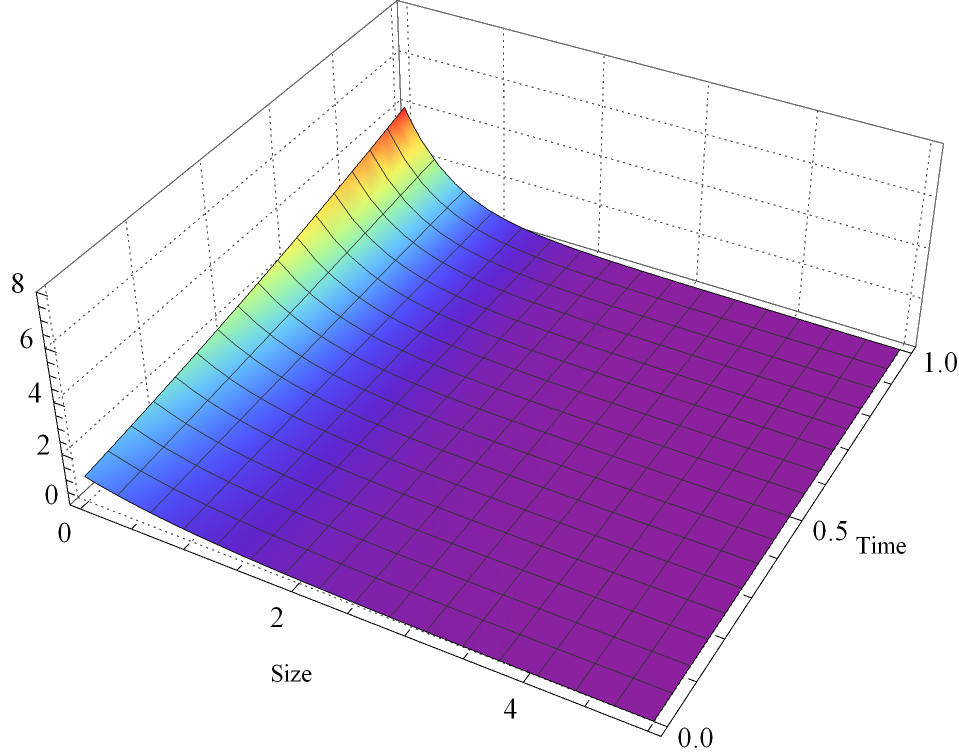}}}
	\subfigure[ODM Concentration   ($\psi_{10}$)]{\includegraphics[width=0.30\textwidth,height=0.25\textwidth]{{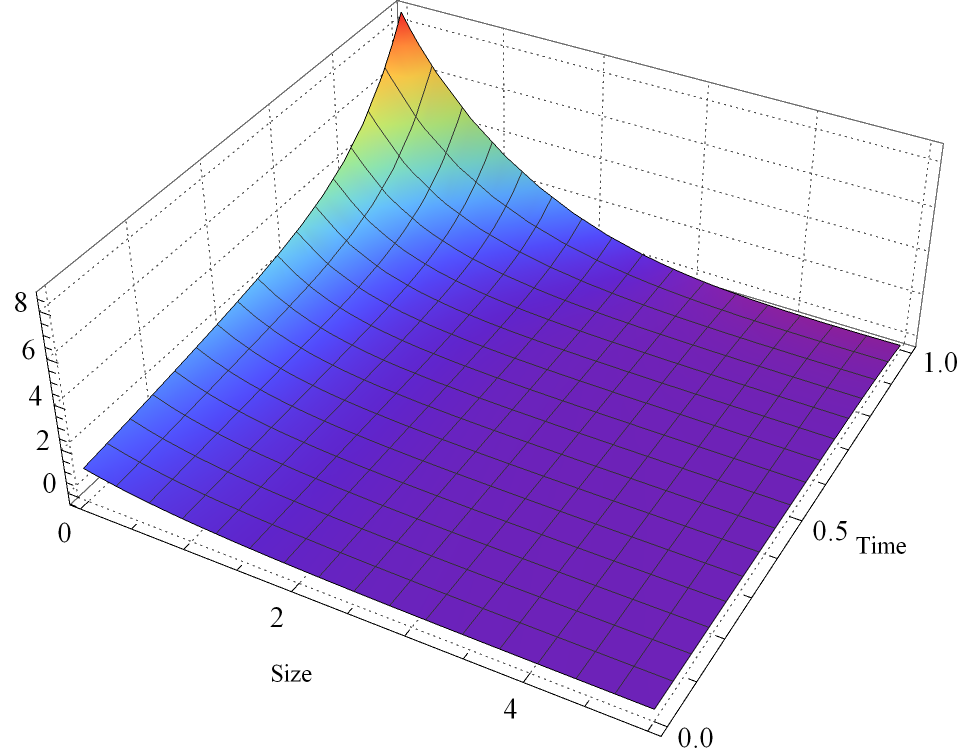}}}
	\subfigure[Exact Concentration]{\includegraphics[width=0.30\textwidth,height=0.25\textwidth]{{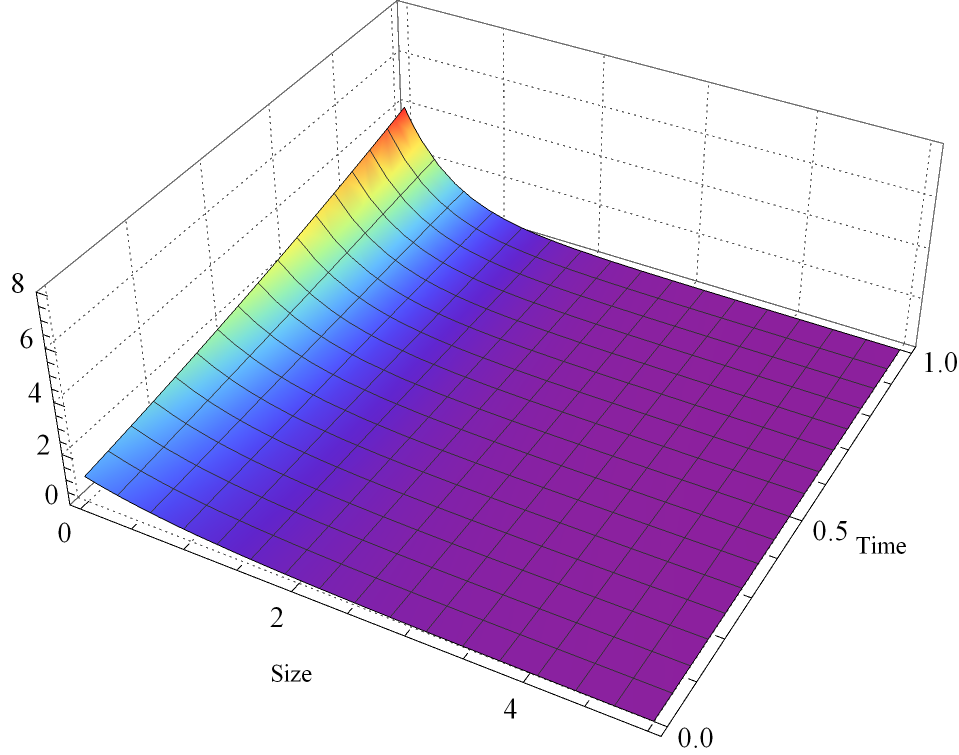}}}
	\caption{ Concentration plots at time $\varsigma$= 1 }
	\label{fig1.1}
	\end{figure}
	 \renewcommand{\thefigure}{4.1(D)}
	\begin{figure}[htb!]
	\centering
	\subfigure[ VIM error]{\includegraphics[width=0.35\textwidth,height=0.255\textwidth]{{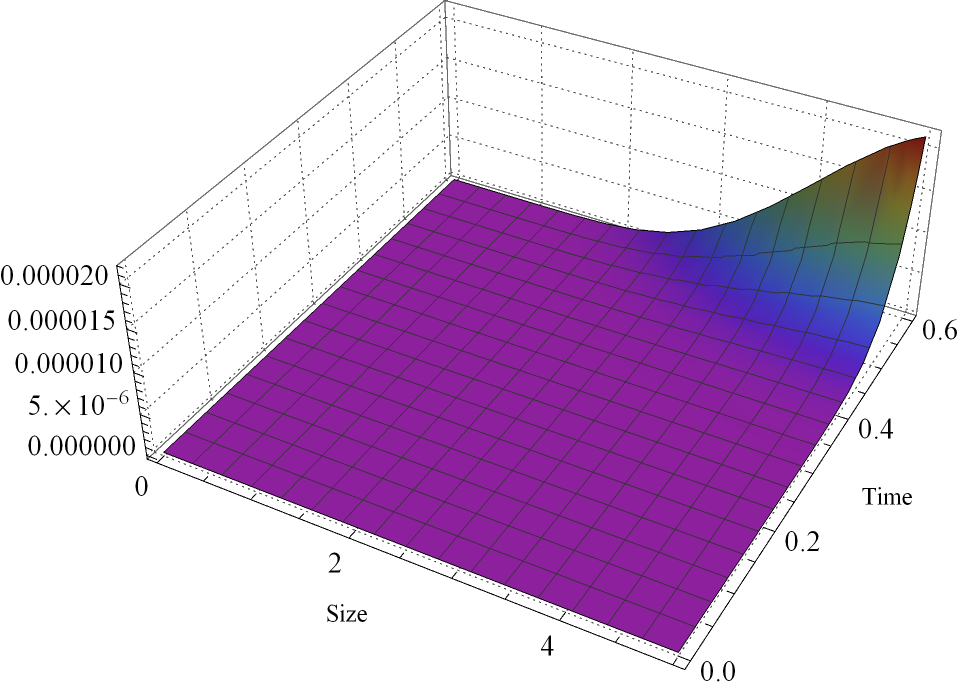}}}
	\subfigure[ ODM error]{\includegraphics[width=0.35\textwidth,height=0.255\textwidth]{{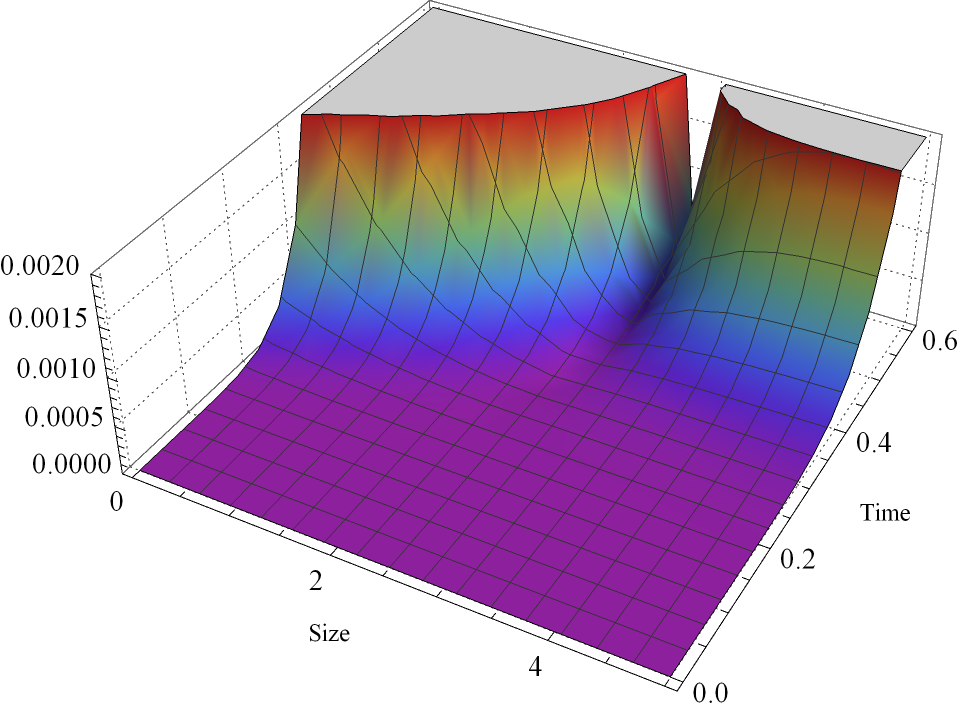}}}
	\caption{ Absolute error between exact solution and  $\varphi_{10}$ as well as $\psi_{10}$}
	\label{fig2}
	\end{figure}
	 \begin{table}[h]
	    \centering
	\begin{tabular}{ |p{0.6cm}|p{1.5cm}|p{2.2cm}|p{2.2cm}|p{2.2cm} |p{2.2cm}|}
	 \hline
	\multirow{2}{*} {$t$}  & \multirow{2}{*} {Methods}  & \multicolumn{4}{|c|}{Number of terms} \\
	 \cline{3-6}
	 & &4 &6 &8 &10 \\ \hline
	0.1 & VIM & 1.8076$\times10^{-6}$ & 1.1147$\times10^{-8}$ & 7.6098$\times10^{-11}$ &4.0914$\times10^{-13}$ \\
	 & ODM & 5.7102$\times10^{-4}$ & 6.5328$\times10^{-5}$ & 6.8206$\times10^{-6}$ &6.6708$\times10^{-7}$ \\\hline
	0.2& VIM  &5.4645$\times10^{-5}$ & 1.3350$\times10^{-6}$&3.6569$\times10^{-8}$ &7.9085$\times10^{-10}$  \\
	&  ODM  &5.3592$\times10^{-3}$ & 1.2903$\times10^{-3}$&2.7836$\times10^{-4}$ &5.7827$\times10^{-5}$  \\\hline
	0.3& VIM &3.9284$\times10^{-4}$ &2.1424$\times10^{-5}$ &1.3232$\times10^{-6}$ &6.4741$\times10^{-8}$  \\
	&  ODM &2.0709$\times10^{-2}$ &7.8754$\times10^{-3}$ &2.7180$\times10^{-3}$ &9.4369$\times10^{-4}$  \\\hline
	0.4 & VIM & 1.5704$\times 10^{-3}$& 1.5126$\times 10^{-4}$ & 1.6630$\times 10^{-5}$ & 1.4540$\times 10^{-6}$\\
	 &  ODM & 5.5200$\times 10^{-2}$& 2.9454$\times 10^{-2}$ & 1.4597$\times 10^{-2}$ & 7.5082$\times 10^{-3}$\\\hline
	0.5& VIM  & 4.5555$\times 10^{-3}$& 6.8176$\times 10^{-4}$ & 1.1722$\times 10^{-4}$ & 1.6092$\times 10^{-5}$\\
	& ODM & 1.1958$\times 10^{-1}$& 8.3816$\times 10^{-2}$ & 5.5852$\times 10^{-2}$ & 3.9088$\times 10^{-2}$\\\hline
	0.6& VIM  & 1.0795$\times 10^{-2}$& 2.3154$\times 10^{-3}$ & 5.7361$\times 10^{-4}$ & 1.1389$\times 10^{-4}$\\
	&  ODM  & 2.2671$\times 10^{-1}$& 2.0006$\times 10^{-1}$ & 1.7113$\times 10^{-1}$ & 1.5330$\times 10^{-1}$\\\hline
	\end{tabular}
	 \caption{Error distribution at different level of time for $n$=4,6,8,10}
	    \label{tab:my_label1}
	\end{table}
	\begin{table}
	 \centering
	\begin{tabular}{ |p{1cm}|p{1cm}|p{1cm}|p{1cm} |p{1cm}|} \hline
	$\gamma_{5}$ & $\gamma_{6}$ &  $\gamma_{7}$ & $\gamma_{8}$ & $\gamma_{9}$ \\ \hline
	0.4513 & 0.4631  & 0.4299 & 0.5625 &0.4142 \\ \hline
	\end{tabular}
	 \caption{Different values of $\gamma_{i}$}
	    \label{tab:my_labe111}
	\end{table}
	
	 \renewcommand{\thefigure}{4.1(E)}
	\begin{figure}[htb!]
	\centering
	\subfigure[Zeorth Moment]{\includegraphics[width=0.30\textwidth,height=0.25\textwidth]{{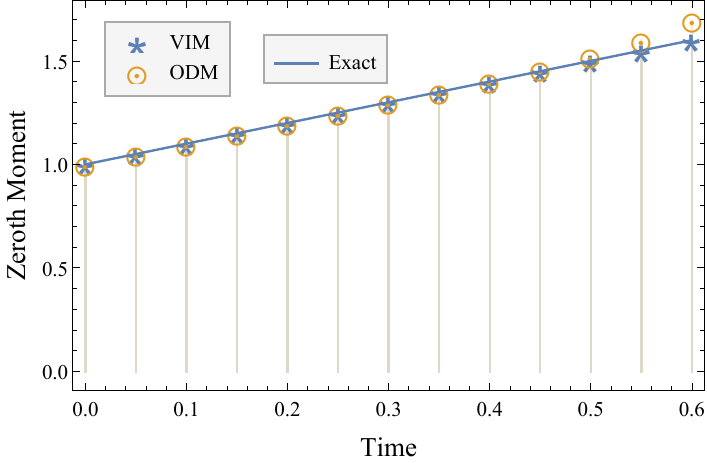}}}
	\subfigure[First Moment]{\includegraphics[width=0.30\textwidth,height=0.25\textwidth]{{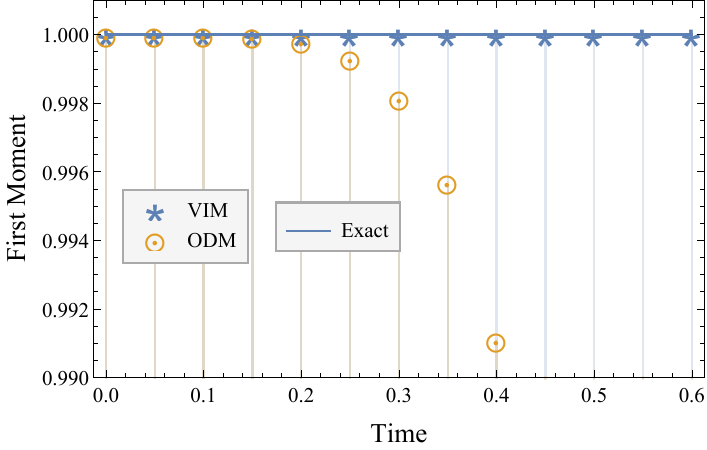}}}
	\subfigure[Second Moment]{\includegraphics[width=0.30\textwidth,height=0.25\textwidth]{{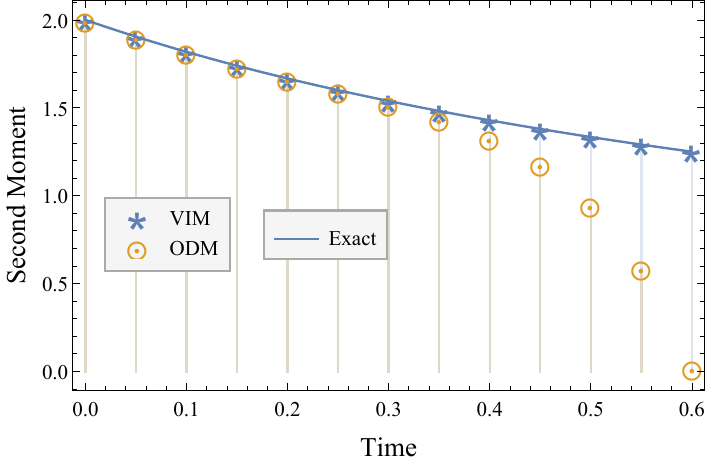}}}
	\caption{Moments comparision: VIM, ODM and Exact  }
	\label{fig3}
	\end{figure}
	
	Proceeding further, the method's efficacy is also affirmed through the comparison between approximated moments of VIM and ODM with the exact moments. Fig.4.1(E)(a) represents the zeroth moment of the approximated solutions $\varphi_{10}$ and $\psi_{10}$ along with the exact zeroth moment. It authenticates that the total number of particles increases as time progresses and the ODM moment commences to slip away from VIM and precise ones around $\varsigma=0.5$. In addition, the first and second moments in Figs.4.1(E)(b) and 4.1(E)(c) are shown and these graphs display that moments using 10-term ODM solutions are not sufficient enough to predict the exact ones in either case as it starts deviating soon after the simulation begins. Notice that a 10-term ODM solution was good to approximate the concentration function for time 0.6 in Fig.4.1(B). So, it is not obvious that if a method delivers good accuracy with solutions, it can also provide similar behaviour with the moments. From these figures, it is certain that VIM moments using $\varphi_{10}$ have better approximation with the exact. The extraction of the above results lies in the significance of VIM over ODM in solving the CBE.
	
	 \begin{exmp}
		 	 Let us consider  Eqs.(\ref{maineq}-\ref{initial}) with kernels $K(\epsilon,\rho)=\frac{\epsilon\rho}{20}$, $b(\epsilon,\rho,\sigma)=\frac{2}{\rho}$ and initial data  $f^{in}(\epsilon)={\epsilon}^2\exp(-\epsilon).$ The analytical solution for the concentration is hard to compute, however, the precise formulations for the zeroth and the first moments are $$M_0(\varsigma)=2+\frac{9\varsigma}{5},\quad \text{and}\quad M_1(\varsigma)=6,$$ respectively. These moments can be generated by multiplying Eq.(\ref{maineq}) by 1 and $\epsilon$ as well as integrating from 0 to $\infty$ over $\epsilon$. Substituting these kernel parameters and initial condition in Eq.(\ref{cvim2}), we have  
	\end{exmp}	
		 	\begin{align*}
		 		 A[f]=&\int_{0}^{\varsigma}\biggl(-\frac{\partial f_k(\tau,\epsilon)}{\partial \tau}+\int_0^\infty\int_{\epsilon}^{\infty} \frac{\sigma}{10}f_k(\tau,\rho)f_k(\tau,\sigma)\,d\rho\,d\sigma \nonumber -\int_{0}^{\infty}\frac{\epsilon\rho}{20}f_k(\tau,\epsilon)f_k(\tau,\rho)\,d\rho\biggl) d \tau,
		 		 \end{align*}
		 		 for the VIM, and so, the following first few terms of the series solution are computed using Eq.(\ref{solutionterm1})
		 		 \begin{align*}
		 		 f_{0}(\varsigma,\epsilon)=&{\epsilon}^2e^{-\epsilon},\quad
		 		      f_{1}(\varsigma,\epsilon)=\frac{3}{10}e^{-\epsilon}\varsigma(4+\epsilon(4-(-2+\epsilon)\epsilon)),\\
		 		       f_{2}(\varsigma,\epsilon)=&\frac{9}{200}e^{-\epsilon}{\varsigma}^{2}(12+{\epsilon}^{2}(-6+(-4+\epsilon)\epsilon)),\\
		 		    f_{3}(\varsigma,\epsilon)=&-\frac{9e^{-\epsilon}{\varsigma}^{3}\epsilon(36+\epsilon(12+\epsilon(-6+(-6+\epsilon)\epsilon)))}{2000}, \\
		 		    f_{4}(\varsigma,\epsilon)=&\frac{27e^{-\epsilon}{\varsigma}^{4}{\epsilon}^{2}(72+(-8+\epsilon)(-2+\epsilon)\epsilon(2+\epsilon))}{80000},\\
		 		    \intertext{and}
		 		    f_{5}(\varsigma,\epsilon)=&-\frac{81e^{-\epsilon}{\varsigma}^{5}{\epsilon}^{3}(120+60\epsilon-10{\epsilon}^3+{\epsilon}^{4})}{4000000}.
		 		 \end{align*} 
		 		The subsequent terms may be calculated using Eq.(\ref{solutionterm1}) and with the help of MATHEMATICA. For the numerical comparison, we have taken the results up to 10-terms solution.\\
		 		
	Furthermore, the ODM technique is being applied to this example to get $C(\epsilon)=-\frac{3\epsilon}{10}$ using Eq.(\ref{codm2}) and co-operation of Eq.(\ref{codm1}) as well as Eqs.(\ref{f1}-\ref{fk}), the following series terms are obtained
		 		\begin{align*}
		 		 f_{0}(\varsigma,\epsilon)&=e^{-\epsilon},\quad
		 			      f_{1}(\varsigma,\epsilon)=\frac{3}{10} \varsigma e^{-\epsilon} (\epsilon (4-(\epsilon-2) \epsilon)+4),\\
		 			       f_{2}(\varsigma,\epsilon)&=-\frac{9}{100}  {\varsigma}^2 e^{-\epsilon} \left(\epsilon \left({\epsilon}^2+\epsilon-2\right)-6\right),\\
		 			    f_{3}(\varsigma,\epsilon)&=\frac{1}{1000}((3 e^{-\epsilon} {\varsigma}^2 (4 \varsigma (-10 + \epsilon (-19 + (-11 + \epsilon) \epsilon)) + 
		 			       15 \epsilon (-4 + \epsilon (-4 + (-2 + \epsilon) \epsilon))))), \\
		 			    f_{4}(\varsigma,\epsilon)&=\frac{1}{20000}(e^{-\epsilon} \left(60 {\varsigma}^3 (\epsilon (\epsilon (\epsilon (6 \epsilon-1)+20)+22)+40)-9 {\varsigma}^4 (\epsilon (\epsilon (\epsilon (15 \epsilon-94)+2)+260)+480)\right)).
		 		\end{align*}	 
		Thanks to MATHEMATICA, it is possible to calculate  the  additional higher terms of the series solution $\psi_{14}$. As the precise concentration function is not there, Fig.4.2(A) plots the absolute difference of consecutive iteration terms at time $\varsigma=1$ and the simulations are summarized for VIM and ODM. It can be observed that the behaviour of the differences is diminishing and  $|f_{10}-f_{9}|$ is near to zero in VIM while $|f_{14}-f_{13}|$ in ODM is still higher than this. Therefore, the series solution is truncated to 10-term $(\varphi_{10})$ in case of VIM and  a 14-term series solution is taken in ODM to anticipate more favourable outcomes. The estimated concentration functions derived by VIM ($\varphi_{10}$) and ODM ($\psi_{14}$) for time 1.5 are shown in Fig.4.2(B), and they are  comparable. The convergence of the VIM solution $(\varphi_{10})$ is further validated by providing the values of  $\gamma_i$ parameters at time 1.5 in Table 3.
			\begin{table}
		 \centering
		\begin{tabular}{ |p{1cm}|p{1cm}|p{1cm}|p{1cm} |p{1cm}|} \hline
		$\gamma_{5}$ & $\gamma_{6}$ &  $\gamma_{7}$ & $\gamma_{8}$ & $\gamma_{9}$ \\ \hline
		0.4212 & 0.4317  & 0.4381 & 0.4323 & 0.3525\\ \hline
		\end{tabular}
		 \caption{Different values of $\gamma_{i}$ at time 1.5}
		    \label{tab:my_labe1111}
		\end{table}
			\renewcommand{\thefigure}{4.2(A)}
		\begin{figure}[htb!]
		\centering
				\subfigure[ VIM  Coefficients]{\includegraphics[width=0.35\textwidth,height=0.255\textwidth]{{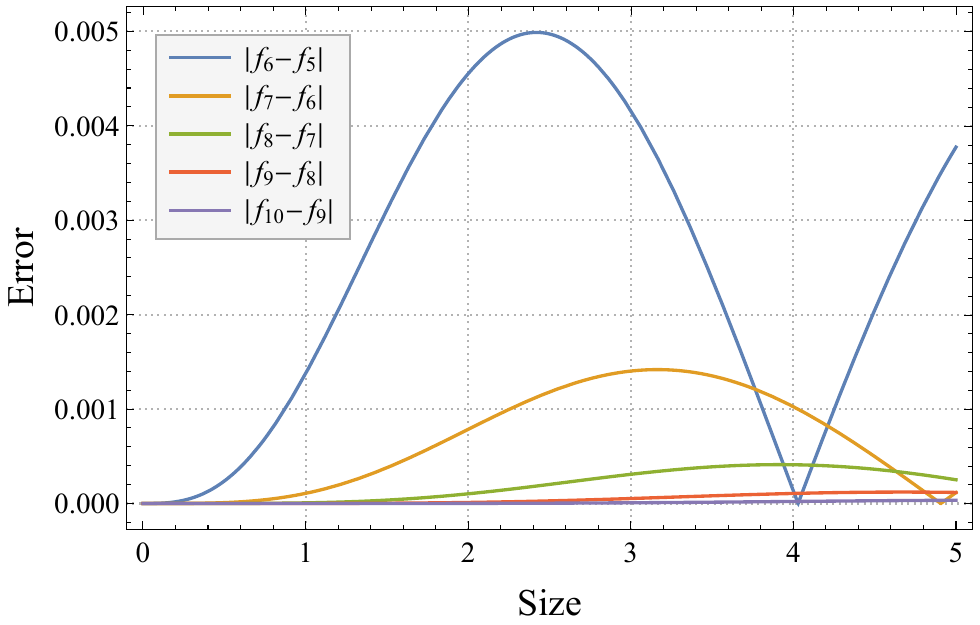}}}
				\subfigure[ ODM Coefficients ]{\includegraphics[width=0.35\textwidth,height=0.255\textwidth]{{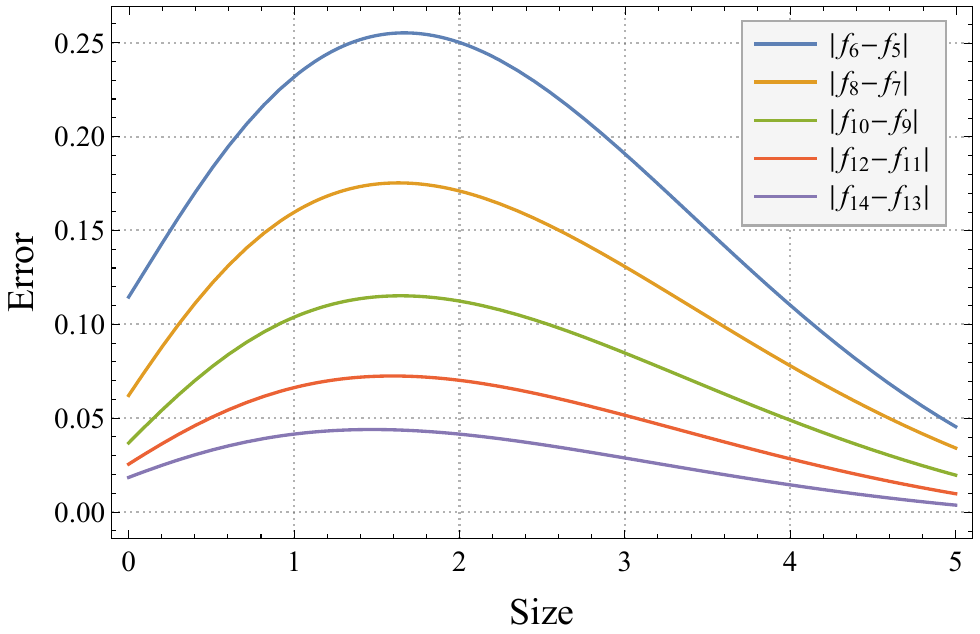}}}
				 	 	 	\caption{Absolute difference of series coefficients at time 1.0}
				 	 	 	\label{fig10}
					 	\end{figure} 
	 	 	  Further, to justify the novelty of our proposed schemes, the zeroth and the first moments of VIM using $\varphi_{10}$ and ODM using $\psi_{14}$ are plotted against the analytical moments in Fig.4.2(C). It is observed that the VIM results exactly overlap with the exact moments, whereas the ODM results start over-predicting and lower-predicting around time $\varsigma=1.2$. Observably, VIM yields findings that resemble analytic solutions. Therefore, one can expect the accurate result for the second moment as well by VIM and hence,  given in Fig.4.2(C)(c). The second moment of ODM missteps away from VIM around $\varsigma=1.2.$ The overall results authenticate that the VIM technique supplies better estimation than ODM. 
			  \renewcommand{\thefigure}{4.2(B)}
		\begin{figure}[htb!]
		\centering
	\subfigure[ VIM solution]{\includegraphics[width=0.35\textwidth,height=0.255\textwidth]{{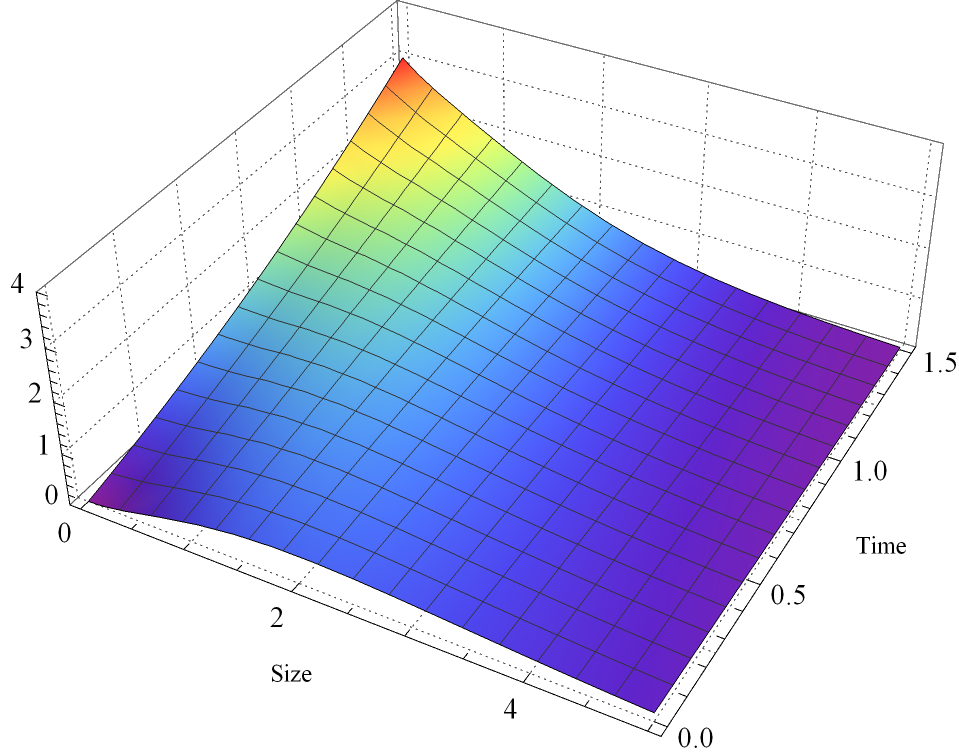}}}
	\subfigure[ ODM solution ]{\includegraphics[width=0.35\textwidth,height=0.255\textwidth]{{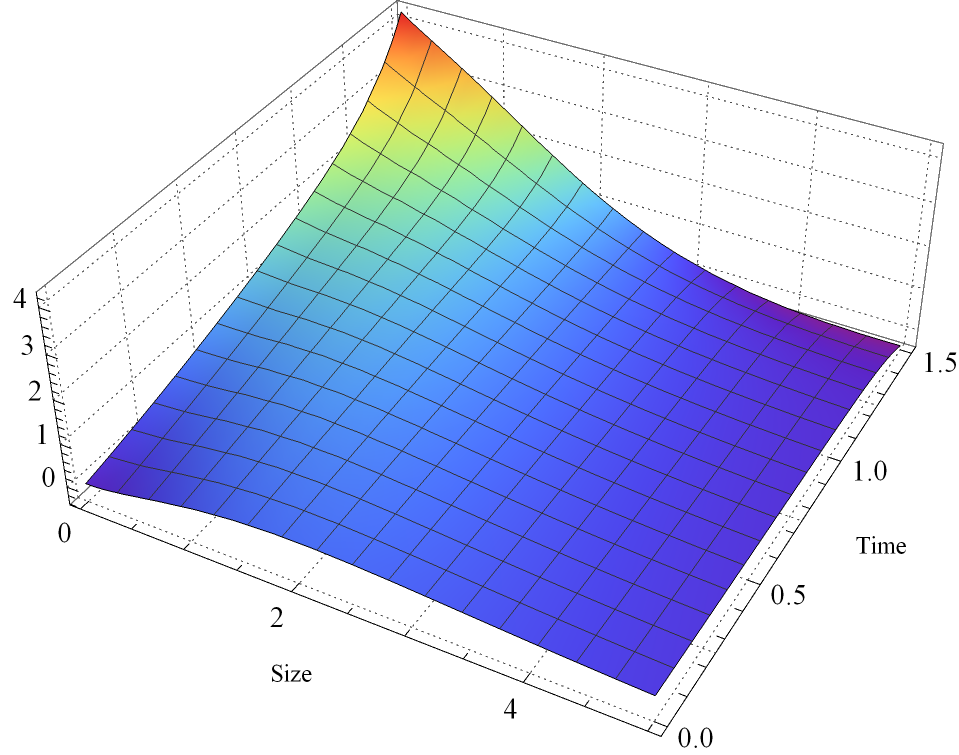}}}
					 	\caption{Plots of series solution with $\varphi_{10}$ and $\psi_{14}$ at time 1.5}
					 	\label{fig11}
					 	\end{figure} 
					 	
		\renewcommand{\thefigure}{4.2(C)}
		\begin{figure}[htb!]
		\centering
	\subfigure[Zeorth Moment]{\includegraphics[width=0.3\textwidth,height=0.25\textwidth]{{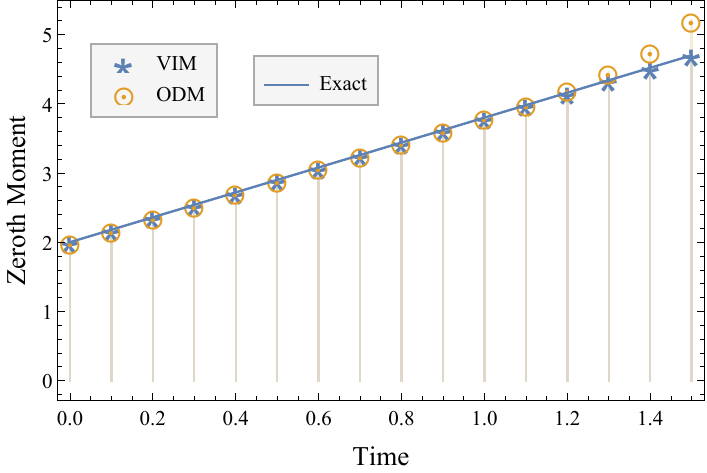}}}
		\subfigure[First Moment]{\includegraphics[width=0.3\textwidth,height=0.25\textwidth]{{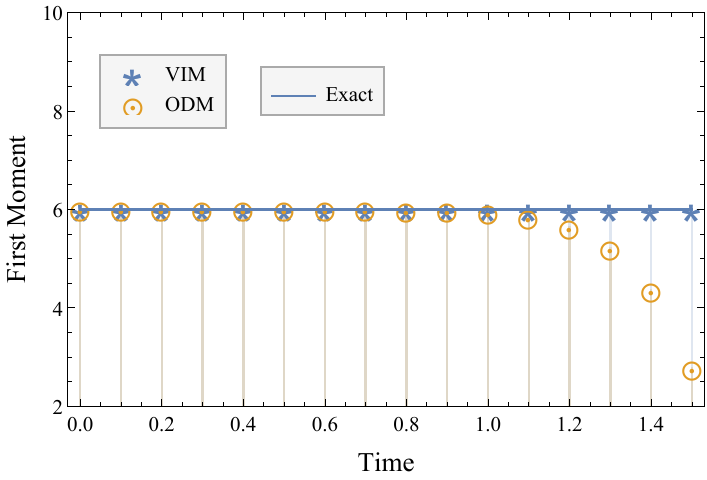}}}
	  \subfigure[Second Moment]{\includegraphics[width=0.3\textwidth,height=0.25\textwidth]{{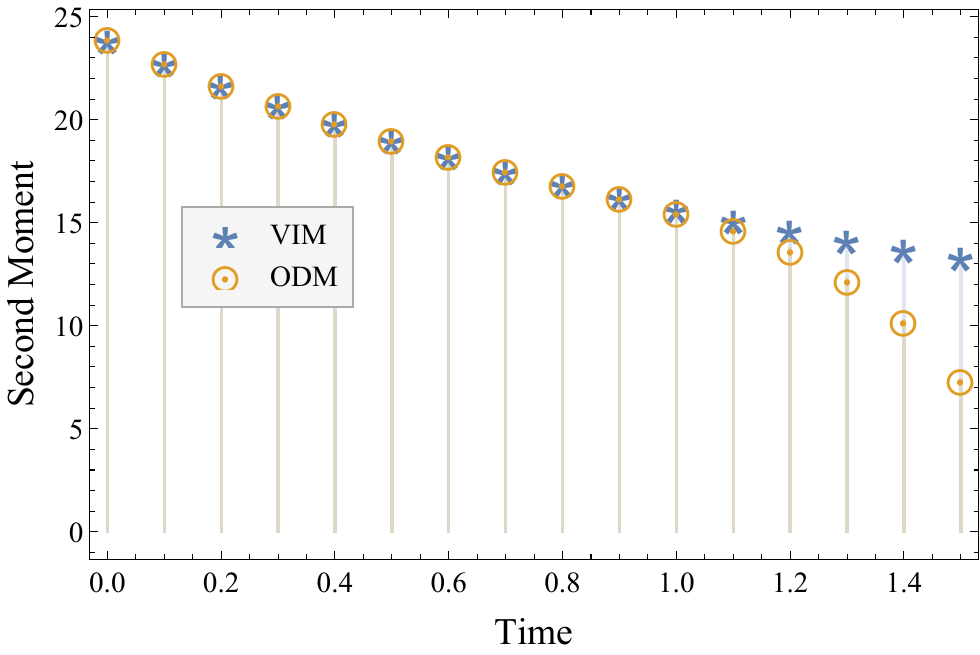}}}
		\caption{Moments comparision: VIM ($\varphi_{10}$), ODM ($\psi_{14}$) and Exact at time 1.5 }
					 			 	 \label{fig9}
					 			 	 \end{figure}
					 			 	 	\begin{exmp}
					 			 	 	Consider again Eqs.(\ref{maineq}-\ref{initial}) with $K(\epsilon,\rho)=1$ and $b(\epsilon,\rho,\sigma)=\delta(\epsilon-0.4\rho)+\delta(\epsilon-0.6\rho)$ along with initial data $f^{in}(\epsilon)=\exp(-\epsilon)$. The exact solution for the concentration does not exist in the literature, but it is easy to determine the exact expression for the first three moments, those are
					 			 	 $$M_0(\varsigma)=\frac{1}{1-\varsigma},\quad  M_1(\varsigma)= 1\quad \text{and}\quad M_2(\varsigma)=2(1-\varsigma)^{0.48}.$$ So, exclusively an approximate solution can be obtained via VIM and ODM. The VIM approach (Eqs.(\ref{cvim2}) and (\ref{solutionterm1})) offers the operator $A$ 
					 			 	 	\end{exmp}
					 			 	 	 \begin{align*}
					 			 	 		 A[f]=&\int_{0}^{\varsigma}\biggl(-\frac{\partial f_k(\tau,\epsilon)}{\partial \tau}+\int_0^\infty\int_{\epsilon}^{\infty} \big(\delta(\epsilon-0.4\rho)+\delta(\epsilon-0.6\rho)\big)f_k(\tau,\rho)f_k(\tau,\sigma)\,d\rho\,d\sigma  \\ &-\int_{0}^{\infty}f_k(\tau,\epsilon)f_k(\tau,\rho)\,d\rho\biggl) d \tau,
					 			 	 		 \end{align*}
					 			 	 		with the following series terms
					 			 	 		  \begin{align*}
					 			 	 		 	 f_{0}(\varsigma,\epsilon)=& e^{-\epsilon},\quad 
					 			 	 		 	      f_{1}(\varsigma,\epsilon)=\epsilon\Big(-e^{-\varsigma}+1.67e^{-1.67\varsigma}\theta(0.67\varsigma)+2.5e^{-2.5\varsigma}\theta(1.5\varsigma)\Big),\\
					 			 	 		 	       f_{2}(\varsigma,\epsilon)=&-0.5e^{-14.19\varsigma}{\epsilon}^{2}\Big(e^{13.19\varsigma}+\theta(0.67\varsigma)(-2.78e^{11.42\varsigma}\theta(1.11\varsigma)\\
					 			 	 		 	       &-4.17e^{10.03\varsigma}\theta(2.5\varsigma))
					 			 	 		 	        +\theta(1.5\varsigma)(-4.17e^{10.03\varsigma}\theta(1.67\varsigma)-6.25e^{7.944\varsigma}\theta(3.75\varsigma))\Big),
					 			 	 		 	 \end{align*}
					 			 	 		 	 where $\delta$ and $\theta$ are Dirac's delta and Heaviside step functions, respectively.\\
					 			 	 
					 			 	 		For the implementation of ODM, having $C(\epsilon)=-1$ for this case, the following series terms are obtained using Eqs.(\ref{codm2}) and (\ref{f1}-\ref{fk})
					 			 	 		 	\begin{align*}
					 			 	 		 	 f_{0}(\varsigma,\epsilon)=&e^{-\epsilon},\quad
					 			 	 		 		      f_{1}(\varsigma,\epsilon)=\varsigma(-e^{-\epsilon}+1.67e^{-1.67\epsilon}\theta[0.67\epsilon]+2.5e^{-2.5\epsilon}\theta[1.5\epsilon]),\\
					 			 	 		 		       f_{2}(\varsigma,\epsilon)=&-0.5e^{-14.19\epsilon}{\varsigma}^2\Big(e^{13.19\epsilon}+\theta[0.67\epsilon]\Big(-2.78e^{11.42\epsilon}\theta[1.11\epsilon]\\
					 			 	 		 		       &-4.17e^{10.03\epsilon}\theta[2.5\epsilon]\Big)+\theta[1.5\epsilon]\Big(-4.17e^{10.03\epsilon}\theta[1.67\epsilon]-6.25e^{7.94\epsilon}\theta[3.75\epsilon]\Big)\Big).
					 			 	 		 	\end{align*}
					 			 	 		 	 The third term is not easy to write due to lengthy expressions in both situations, and in fact, it is difficult to compute the higher series terms by MATHEMATICA due to the complexity of the model and kernels. It consumes a significant amount of time and therefore, we consider $\varphi_{3}$ and $\psi_3$ as approximate solutions for the numerical validation. Since, analytical solution is not available for the concentration function, absolute differences between successive series terms are made up to 3-terms, see Fig.4.3(A). Initially, the differences have minor gap but as size increases the error tends to zero. Also, it is clear that the solution using 3-terms series result is reliable enough. In Fig.4.3(B), series solutions $\varphi_{3}$ and $\psi_{3}$ are almost identical at time 0.2.\\
					 			 	 
					 			 	 Further, to see the efficiency of one algorithm over other, the comparison of approximated and analytical moments are made.  In Fig.4.3(C)(a), the zeroth moment of $\varphi_3$ and $\psi_3$ are presented with the exact one at $\varsigma$=0.6. The ODM moment does not agree with the exact one after $\varsigma=$ 0.5. In contrast, VIM provides better accuracy at this time. As expected, mass is conserved and this is also visualized from the results for the first moment using VIM and ODM in Fig.4.3(C)(b). Both methods provide suitable precision with the exact mass at $\varsigma=0.6$. Finally, Fig.4.3(C)(c) depicts the second moments of VIM and ODM with VIM outperforming ODM in estimation. Hence, it is concluded that the VIM moments are closest to the actual ones.
					 			 	 		 	  \renewcommand{\thefigure}{4.3(A)}
					 			 	 		 	 	 	\begin{figure}[htb!]
					 			 	 		 	 	 	\centering
					 			 	 		 	 	 	\subfigure[ VIM  coefficients]{\includegraphics[width=0.35\textwidth,height=0.255\textwidth]{{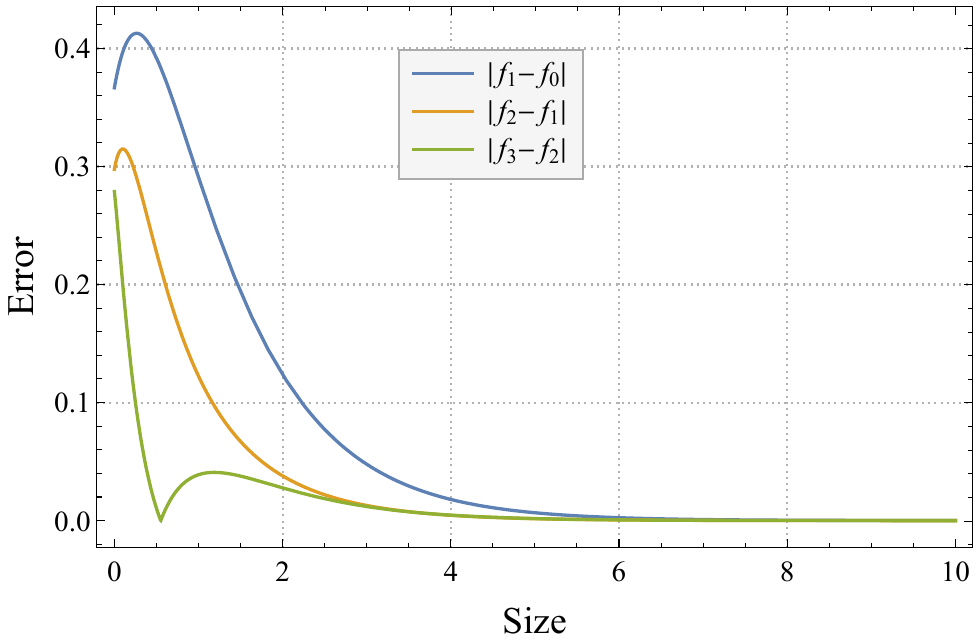}}}
					 			 	 		 	 	 	\subfigure[ ODM coefficients ]{\includegraphics[width=0.35\textwidth,height=0.255\textwidth]{{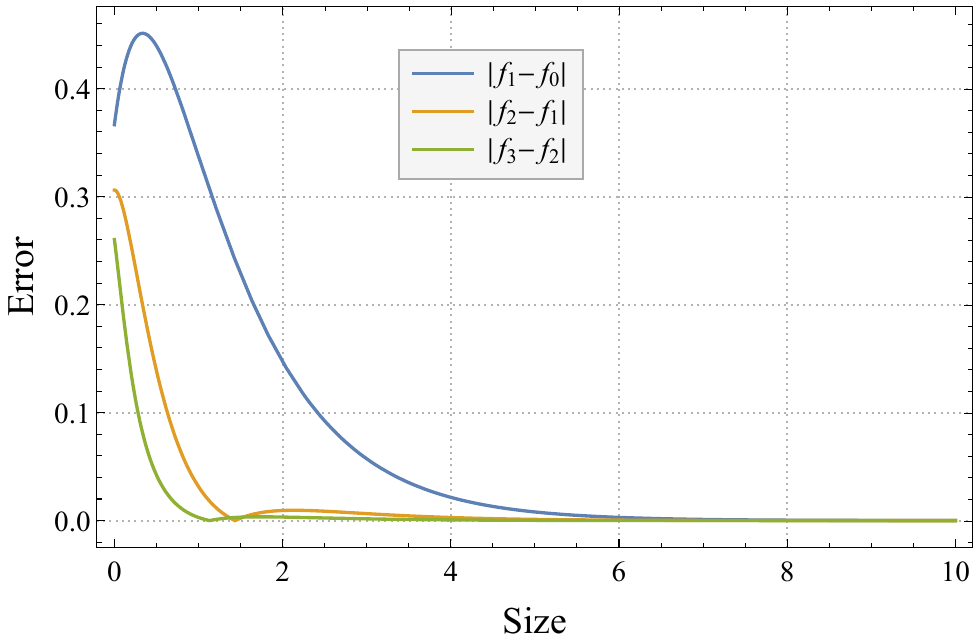}}}
					 			 	 		 	 	 	\caption{Absolute difference of series coefficients at time 0.2}
					 			 	 		 	 	 	\label{fig6}
					 			 	 		 	 	 	\end{figure} 
					 			 	 		 	  \renewcommand{\thefigure}{4.3(B)}
					 			 	 		 	\begin{figure}[htb!]
					 			 	 		 	\centering
					 			 	 		 	\subfigure[ VIM solution]{\includegraphics[width=0.35\textwidth,height=0.255\textwidth]{{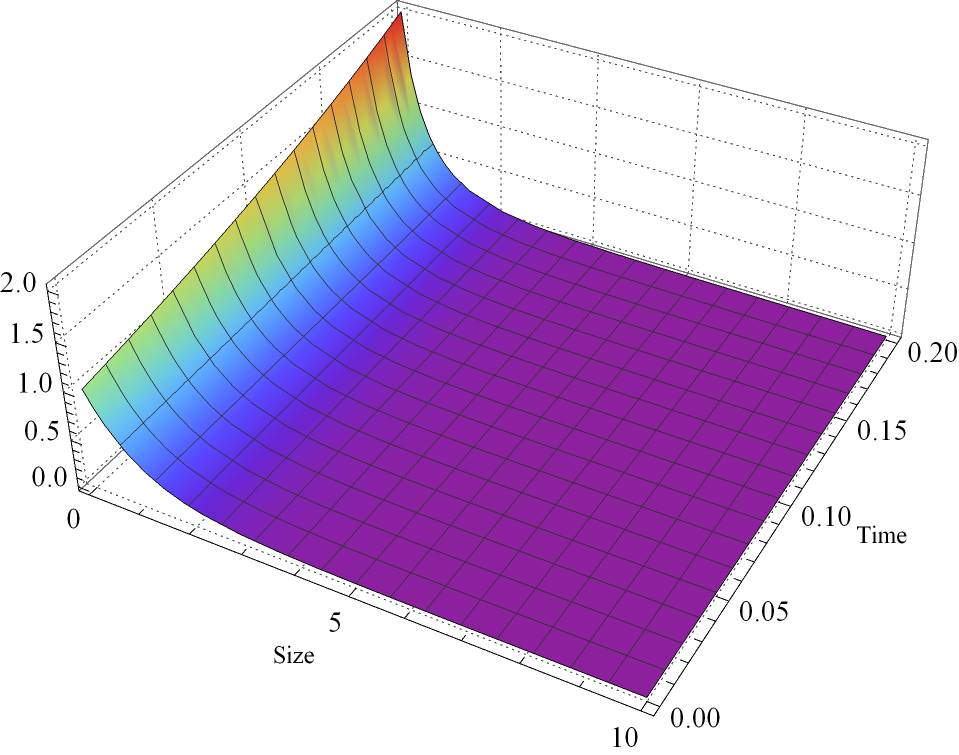}}}
					 			 	 		 	\subfigure[ ODM solution ]{\includegraphics[width=0.35\textwidth,height=0.255\textwidth]{{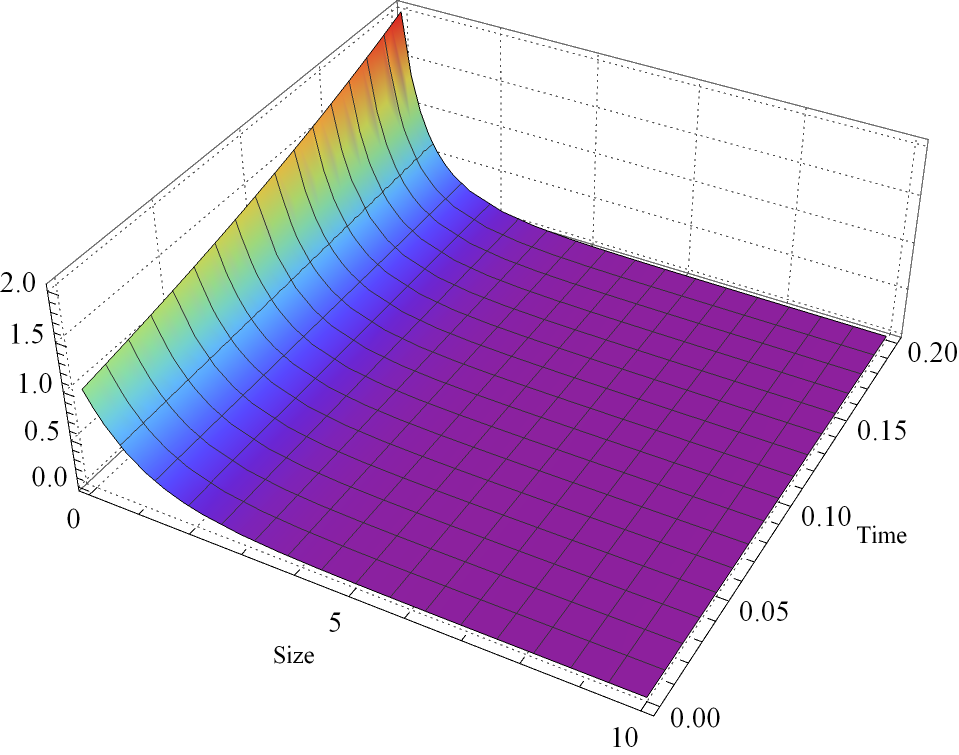}}}
					 			 	 		 	\caption{Plots of series solution with $\varphi_{3}$ and $\psi_{3}$}
					 			 	 		 	\label{fig7}
					 			 	 		 	\end{figure} 
					 			 	 		
					 			 	 		 	  \renewcommand{\thefigure}{4.3(C)}
					 			 	 		 	 \begin{figure}[htb!]
					 			 	 		 	 \centering
					 			 	 		 	 \subfigure[Zeorth Moment]{\includegraphics[width=0.3\textwidth,height=0.25\textwidth]{{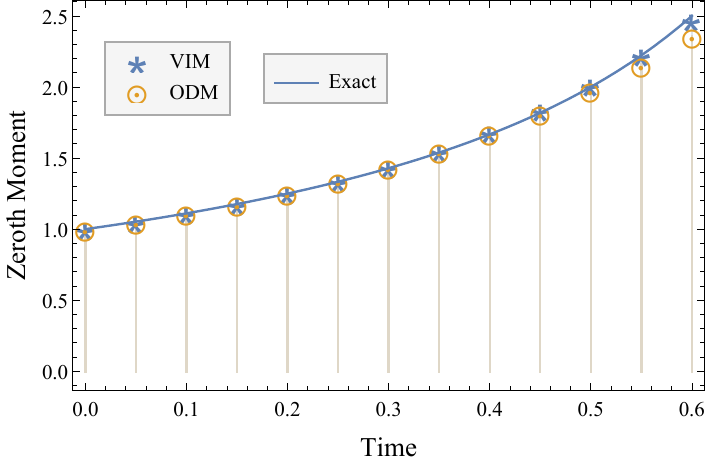}}}
					 			 	 		 	 \subfigure[First Moment]{\includegraphics[width=0.3\textwidth,height=0.25\textwidth]{{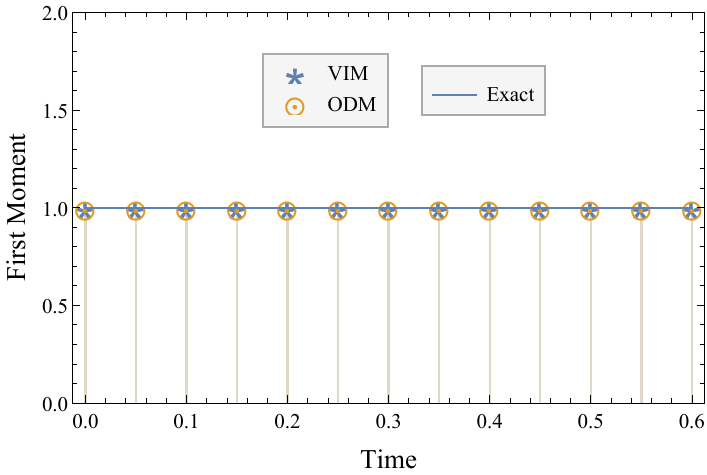}}}
					 			 	 		 	 \subfigure[Second Moment]{\includegraphics[width=0.3\textwidth,height=0.25\textwidth]{{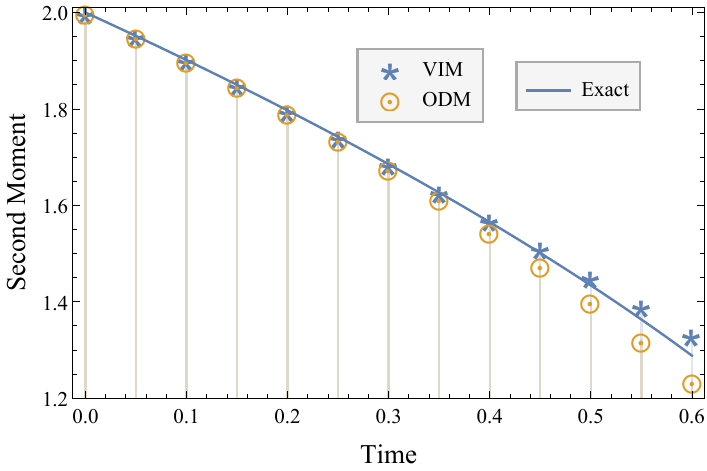}}}
					 			 	 		 	 \caption{Moments comparision: VIM ($\varphi_3$), ODM ($\psi_3$) and Exact at time 0.6 }
					 			 	 		 	 \label{fig9}
					 			 	 		 	 \end{figure}		 	 	 	
					 			 	 	\section{Conclusions}\label{conclusions} Semi-analytical VIM and ODM techniques were implemented on the non-linear collision-induced breakage equation in this work. These methods provided a recursive algorithm for calculating the closed form of analytical solution or an approximative results using finite term series approximations. The theory of convergence result for VIM was taken from \cite{odibat2010study} while in ODM, the contraction mapping theorem was used to show the series convergence in detail. The discussion was reliable enough to estimate the maximum absolute truncated error. The applicability and accuracy of these methods were shown by considering three different test cases. The approximated solutions and various moments were compared against the analytical solutions and error graphs were provided. Interestingly, in one case, we got the closed form of solution via VIM. In each example, VIM's solution demonstrated superior long-term compliance with analytical requirements, whereas ODM's solution failed to do so. Moreover, the outcomes ensured that VIM is superior to ODM. It was also an important observation to mention that the moment plots were more precise than concentration plots over extended time periods.

	\section{Acknowledgments}
	The first author work is supported by CSIR India, which provides the PhD fellowship and the file No. is 1157/CSIR-UGC NET June 2019.
		 	
\bibliography{semiref}
\bibliographystyle{ieeetr}
				\end{document}